\documentclass[11pt]{article}
\usepackage{mathrsfs}
\usepackage{amssymb}
\usepackage{amsmath}
\usepackage{amsthm}
\usepackage{amscd}
\usepackage{epstopdf}
\usepackage{stmaryrd}
\usepackage{stmaryrd}
\usepackage{enumerate}
\usepackage{hyperref}

\newtheorem{Theorem}{Theorem}[section]
\newtheorem{Claim}[Theorem]{Claim}
\newtheorem{Lemma}[Theorem]{Lemma}
\newtheorem{Proposition}[Theorem]{Proposition}
\newtheorem{Definition}{Definition}
\newtheorem{Corollary}[Theorem]{Corollary}

\newcommand{\xt}{^{\times 2}}
\newcommand{\x}{^{\times}}
\newcommand{\N}{\mathbb{N}}
\newcommand{\Q}{\mathbb{Q}}
\newcommand{\Lat}{\mathcal{L}}
\newcommand{\Z}{\mathbb{Z}}
\newcommand{\R}{\mathbb{R}}
\newcommand{\C}{\mathbb{C}}
\newcommand{\F}{\mathbb{F}}

\newcommand{\D}{\mathcal{D}}
\newcommand{\Dab}{\mathcal{D}^{(a,b)}}
\newcommand{\Dabs}{\mathcal{D}_{\sol}^{(a,b)}}
\newcommand{\tauab}{\eta^{(a,b)}}
\newcommand{\Fu}{\mathcal{F}}
\newcommand{\Gij}{H}
\newcommand{\Vij}{W}
\newcommand{\Kij}{K_1 \times K_2}
\newcommand{\dxt}{d^{\times}t}

\newcommand{\wij}{\phi_{12}}
\newcommand{\wijp}{\phi_{12,p}}
\newcommand{\Bij}{(\ ,\ )}
\newcommand{\BB}{\langle\ ,\ \rangle}
\newcommand{\B}{(\ ,\ )}
\newcommand{\J}{J}

\DeclareMathOperator{\rs}{rs}
\DeclareMathOperator{\ad}{ad}
\DeclareMathOperator{\h}{ht}
\DeclareMathOperator{\sep}{sep}
\DeclareMathOperator{\Vol}{Vol}
\DeclareMathOperator{\GL}{GL}
\DeclareMathOperator{\SO}{SO}

\DeclareMathOperator{\Res}{Res}
\DeclareMathOperator{\Tr}{Trace}
\DeclareMathOperator{\Gal}{Gal}
\DeclareMathOperator{\SL}{SL}
\DeclareMathOperator{\Inv}{Inv}
\DeclareMathOperator{\Invrs}{Inv^{rs}}
\DeclareMathOperator{\Stab}{\Delta}
\DeclareMathOperator{\Sel}{Sel}
\DeclareMathOperator{\Spec}{Spec}
\DeclareMathOperator{\sol}{sol}
\DeclareMathOperator{\topp}{top}
\DeclareMathOperator{\Id}{I}
\DeclareMathOperator{\Disc}{Disc}
\DeclareMathOperator{\red}{red}
\DeclareMathOperator{\irr}{irr}
\title{2-Selmer groups of hyperelliptic curves with two marked points}
\author{Ananth N.\ Shankar\\ }
\date{\today}

\makeatletter
\def\iddots{\mathinner{\mkern1mu\raise\p@
\vbox{\kern7\p@\hbox{.}}\mkern2mu
\raise4\p@\hbox{.}\mkern2mu\raise7\p@\hbox{.}\mkern1mu}}
\makeatother

\begin{document}
\maketitle
\abstract{We consider the family of hyperelliptic curves over $\Q$ of fixed genus along with a marked rational Weierstrass point and a marked rational non-Weierstrass point. When these curves are ordered by height, we prove that the average Mordell-Weil rank of their Jacobians is bounded above by 5/2. We prove this by showing that the average rank of the 2-Selmer groups is bounded above by 6. We also prove that the average size of the $\phi$-Selmer groups of a family of isogenies associated to this family is exactly 2. }

\section{Introduction}
There has been a lot of recent progress on studying the statistics of
Jacobians and rational points of familes of curves. In \cite{BG},
Manjul Bhargava and Benedict Gross prove that when all odd-degree
hyperelliptic curves over $\Q$ are ordered by height, the average size
of the 2-Selmer group of their Jacobians is bounded by $3$, and the
average rank of the Jacobians is bounded by $3/2$. Using these
results, Bjorn Poonen and Michael Stoll in \cite{PoSto} prove that a
positive proportion of odd-degree hyperelliptic curves over $\Q$ have
exactly one rational point (namely, the Weirstrass point at infinity),
and that this proportion goes to one as the genus goes to infinity. In
\cite{SW}, Arul Shankar and Xiaoheng Wang prove results analogous to
\cite{BG} and \cite{PoSto} for the family of monic even-degree
hyperelliptic curves. Jack Thorne, in \cite{JThorne}, studies the
statistics of the 2-Selmer set in a family of non-hyperelliptic
curves, which is a pointed subet of the 2-Selmer group. He proves that
the average size of the 2-Selmer set is finite. He uses these
statistics to prove that a positive proportion have integral points
everywhere locally, but have no global integral points.

In this work, we prove results about average Selmer sizes for different families of curves.  We recall the definition of $\phi$-Selmer groups, where $\phi: A \rightarrow B$ is an isogeny of abelian varieties over $\Q$. Let $A[\phi]$ denote its kernel. The action of $\Gal(\Q^{\sep}/\Q)$ on the exact sequence 
$$0\rightarrow A[\phi](\overline{\Q}) \rightarrow A(\overline{\Q}) \rightarrow B(\overline{\Q}) \rightarrow 0$$
gives a long exact sequence of the Galois cohomology groups. In particular, there is an injective map $B(\Q)/\phi A(\Q) \rightarrow H^1(\Q,A[\phi])$, where $H^1(\Q, A[\phi])$ is the Galois cohomology group of $\Gal(\overline{\Q}/\Q)$ with coefficients in $A[\phi](\overline{\Q})$. The $\phi$-Selmer group $\Sel_{\phi}$ over $\Q$ is a finite subgroup of $H^1(\Q, A[\phi])$ consisting of elements which locally lie in the images of $B(\Q_v)/\phi A(\Q_v)$, for all completions $\Q_v$ of $\Q$. This definition recovers the classical definitions of the $n$-Selmer groups, by choosing the isogeny $\phi$ to be multiplication by $n$. 

Consider a smooth hyperelliptic curve $C_1$ of genus $m \geq 2$ over $\Q$, with a marked rational Weirstrass point that we denote by $\infty_1$, and a marked rational non-Weirstrass point that we denote by $P_1$. Let $P'_1$ denote the conjugate of $P_1$ under the hyperelliptic involution. Without loss of generality, we may assume that under the natural map $C_1 \rightarrow \mathbb{P}^1$, $\infty_1$ maps to $\infty \in \mathbb{P}^1(\Q)$, and $P_1$ maps to $0 \in \mathbb{P}^1(\Q)$. Therefore, $C_1$ has an affine equation of the form
\begin{equation*}\label{form}
y^2 = x^{2m+1} + a_1 x^{2m} + \hdots a_{2m}x + e^2 = f(x)
\end{equation*}
where $f(x) \in \Q[x]$ is separable over $\Q$, and $e \in \Q^{\times}$. If we assume that $f(x)$ has integral coefficients, and that there is no prime $p$, such that $p^{2i}| a_i$ for all $i$ and $p^{2m+1} | e$, then the equation $y^2 = f(x)$ is unique. Denote the family of such polynomials by $\mathcal{B}$. We define the height of $C_1$ to be 
$$\h(C_1) := \h(f) := \max \{|a_i|^{1/2i}, |e|^{1/2m+1} \}.$$
It follows from this definition that for fixed $X \in \R$, there are finitely many curves with height bounded by $X$. Let $J_1$ denote the Jacobian of $C_1$. The first main result of the paper is
\begin{Theorem}\label{2selm}
When all hyperelliptic curves of a fixed genus $m \geq 2$ having a marked rational Weirstrass point and a marked rational non-Weirstrass point, are ordered by height, the average size of the 2-Selmer groups of their Jacobians is bounded above by 6. 
\end{Theorem}
The family considered in \cite{SW} is the family of hyperelliptic curves with a marked non-Weirstrass point. The additional prescence of the marked Weirstrass point $\infty_1$ in a curve $C_1$ of our family has the consequence of introducing a square-root of the class of $(P_1) - (P'_1)$ in $J_1(\Q)$, where $J_1$ is the Jacobian of $C_1$. Indeed, $(P_1) -( P'_1)$ is twice the class of $(P_1) - (\infty)$. We therefore expect that despite the existence of the extra marked Weirstrass point, the statistics of the 2-Selmer size for our family is the same as the statistics for the family of even hyperelliptic curves. This expectation is supported by the fact that our result agrees with \cite[Theorem 1.2]{SW}.

We prove that for 100\% of curves in our family, the class of $(P_1) - (\infty_1)$ in $J_1(\Q)$ is not divisible by 2. Therefore, the average contribution of $(P_1) - (\infty_1)$ to the 2-rank of $\Sel_2(J_1)$ is 1. Denote the 2-rank by $r_1$. The following inequality holds 100\% of the time: 
$$2(r_1-1) \le 2^{r_1 -1} =\frac {\#\Sel_2(J_1)}{2}.$$ 
It follows that the average 2-rank of the 2-Selmer groups of Jacobians of curves in our family is at most $5/2$. Because the 2-Selmer rank is an upper-bound for the Mordell-Weil rank, we obtain
\begin{Corollary}
When all hyperelliptic curves of a fixed genus $m \geq 2$ having a marked rational Weirstrass point and a marked rational non-Weirstrass point are ordered by height, the average rank of the Mordell-Weil group of their Jacobians is bounded above by $5/2$. 
\end{Corollary}

To a curve $C_1$ corresponding to $f(x) \in \mathcal{B}$, we associate two other curves $C_2$ and $C$, where $C_2$ is given by the equation $y^2 = xf(x)$, and $C$ is given by the equation $y^2 = f(x^2)$. We therefore obtain two other familes of hyperelliptic curves as $f$ varies over $\mathcal{B}$. We have $J_1[2] \simeq J_2[2]$ as group schemes over $\Q$, where $J_1$ and $J_2$ are the Jacobians of $C_1$ and $C_2$ respectively. We denote this group scheme by $\Delta$. Therefore, the 2-Selmer groups of $J_1$ and $J_2$ are subgroups of the same group, $H^1(\Q,\Delta)$. Denote their intersection by $\Sel_{(1,2)}(f)$. 
\begin{Theorem}\label{12selm}
The average size of $\Sel_{(1,2)}(f)$, as $f \in \mathcal{B}$ is ordered by height, is equal to $2$. 
\end{Theorem}
The group $\Sel_{(1,2)}$ always contains the identity of $H^1(\Q, \Delta)$, and the image of $(P_1) - (\infty_1) \in J_1(\Q)/2J_1(\Q) \subset H^1(\Q,\Delta)$. Theorem \ref{12selm} implies that 100 \% of the time, $\Sel_{(1,2)}$ contains nothing else. 

There are canonical maps $C \rightarrow C_1$ ($(x,y) \mapsto (x^2,y)$) and $C \rightarrow C_2$ ($(x,y) \mapsto (x^2,xy)$).  The Jacobians of these curves form the exact sequence
$$0 \longrightarrow \Delta \longrightarrow J_1 \times J_2 \xrightarrow{\ \phi\ } J \longrightarrow 0$$
where $J$ denotes the Jacobian of $C$. Since the $\phi$-Selmer group of $J_1 \times J_2 \rightarrow J$ equals $\Sel_{(1,2)}$, we obtain the following corollary to Theorem \ref{12selm}.
\begin{Corollary}
Let notation be as above. When $f \in \mathcal{B}$ is ordered by height, the average size of the $\phi$-Selmer group is $2$. 
\end{Corollary}

The structure of the paper of the paper follows \cite{BG}. In \S2, we consider the representation $n \otimes n$ of the split semisimple group $\SO_n \times \SO_n$, where $n = 2m+1$ is an odd integer. By viewing this representation in the Vinberg setting, a vector in this representation can be viewed as a self-adjoint operator on a $2n$-dimensional vector space, whose characteristic polynomial is of the form 
$$f(x^2) = x^{2n} + a_1x^{2(n-1)} + \hdots a_{n-1}x^2 + e^2.$$
The functions $a_1, \hdots a_{n-1}, e$ are invariant under the action of $\SO_n \times \SO_n$. In fact, these freely generate the ring of $\SO_n \times \SO_n$-invriants. A point in the invariant space is said to be regular semisimple if the corresponding polynomial $f(x^2)$ is separable. Using Thorne's work \cite{Thorne}, we demonstrate the existence of two sections $\kappa_1$ and $\kappa_2$ from the space of invariants to $n\otimes n$. Further, we prove that the orbit of $\kappa_i(c)$ is {\it distinguished} (which we define in \S 2), where $c$ is regular semisimple. 

In \S 3, we prove that the regular semisimple invariants separate geometric $\SO \times \SO$ orbits. Using the language of \cite{AIT}, we describe in \S 4 how geometric orbits break up over arbitrary fields.

In \S 5, we associate two pencils of quadrics to each $\SO_n \times \SO_n$-orbits on $n \otimes n$. The theory developed in \cite{Jerrythesis} realises the Fano variety of these pencils as torsors for $J_1$ and $J_2$, where $J_1$ is the Jacobian of the curve $y^2 = f(x)$, and $J_2$ is the Jacobian of the curve $y^2 = xf(x)$. We prove that there is a bijection between the 2-Selmer group of $J_1[2]$, and rational orbits with these invariants such that the first Fano variety has points over $\Q_v$ for every place $v$. We call these orbits {\it locally soluble orbits}. Similarly, there is a bijection between the intersection of the 2-Selmer groups of $J_1$ and $J_2$, and rational orbits such that both the Fano-varieties have points over $\Q_v$ for all places $v$. We call these orbits \textit{locally (1,2)-soluble orbits}. A crucial ingredient needed to prove Theorems \ref{2selm} and \ref{12selm} is to demonstrate the existence of integral representatives (with a minor technical condition at the place 2) of locally soluble $\SO_n(\Q) \times \SO_n(\Q)$ orbits, which have integral invariants. This is done in \S 6, and the result we prove is stated as Theorem \ref{intorb}. 

Having parameterised the Selmer groups in terms of integral soluble and (1,2)-soluble orbits, we use Bhargava's geometry of numbers techniques (\cite{BQuartic}) to estimate the number of these orbits. In order to do this, in \S 7 we count the number of points inside a fundamental domain for the action of $\SO_n(\Z) \times \SO_n(\Z)$ on $\R^n \otimes \R^n$. This fundamental domain splits into two parts: the main body, which we prove contains a negligible number of distinguished orbits; and the cusp, which we prove contains predominantly distinguished orbits. 

In \S 8, we impose appropriate congruence conditions to pass from integral orbits to locally soluble integral orbits, or to locally (1,2)-soluble integral orbits. In the first case, the main body will contribute on average at most four Selmer elements (based on work in progress of Manjul Bhargava, Arul Shankar and Xiaoheng Wang, we expect that the contribution will be exactly four Selmer elements on average), and the cusp will correspond to the distinguished orbits, which are the marked elements in the Selmer group. This gives that the average size of the 2-Selmer group is bounded by 6, proving Theorem \ref{2selm}. In the second case, we prove that the product of the local densities diverges to zero (Proposition \ref{diverges}), and so the only contribution to the average comes from the cusp. This proves Theorem \ref{12selm}.

In future work, we will use these results to bound the number of rational points of curves in these familes. 

\subsection*{Acknowledgements}
I am very grateful to Benedict Gross for suggesting that I work on this project, and also for numerous useful conversations. I also thank Manjul Bhargava, Chao Li, Arul Shankar, Jack Thorne, Cheng-Chiang Tsai and Xiaoheng Wang for helpful discussions. Finally, I would like to thank Arul Shankar and Xiaoheng Wang for very helpful comments on previous versions of this paper. 
\section{A representation of $\rm{SO_n} \times \rm{SO_n}$}
Let $k$ be a field of characteristic different from $2$. In this section, we consider the action of $\rm{SO}_n \times \rm{SO}_n$ on $n \otimes n$. 

\subsection{Vinberg theory}
The above representation is in the Vinberg setting. Indeed, let $(V_1, Q_1, \epsilon_1)$ be an $n$-dimensional split orthogonal space, with discriminant 1 with respect to the basis $\epsilon_1$ of $\bigwedge^{\topp}(V_1)$, and $(V_2,Q_2,\epsilon_2)$ be an $n$-dimensional split orthogonal space of discriminant $(-1)^n$ (the discriminant is again relative to $\epsilon_2 \in \bigwedge^{\topp}(V_2)$).  Consider the $2n-$dimensional split orthogonal space $V = V_1 \oplus V_2, Q = Q_1 \oplus Q_2$, and the special orthogonal group $G = \rm{SO}(V)$ (the discriminant condition on $(V_2,Q_2)$ is imposed so that the space $(V,Q)$ is split). 

We consider the involution $\theta$ of $\rm{SO}(V)$, given by conjugation by the element $(\Id_n,-\Id_n) \in \GL(V_1) \times \GL(V_2)$. Let $G^{\theta}$ be the subgroup fixed by $\theta$. It is the intersection of $\rm{O}(V_1) \times \rm{O}(V_2)$ with $\rm{SL}(V)$, inside $\rm{GL}(V)$. The involution $\theta$ also acts on the lie algebra $\mathfrak{g}$. Let $\mathfrak{g}_1$ denote the $-1$ eigenspace. It consists of skew self-adjoint operators on $V$ whose diagonal blocks (with respect to the decomposition $V=V_1 \oplus V_2$) are $0$.

The action of $G^{\theta}$ on $\mathfrak{g}$ preserves the eigenspace $\mathfrak{g}_1$. As representations of $G^{\theta}$, $\mathfrak{g}_1 \cong V_1 \otimes V_2$. The isomorphism can be described as follows:  given an element $\alpha \in V_1 \otimes V_2$, we think of it as an operator $T_1:$ $V_1 \rightarrow V_2$ using the bilinear form on $V_1$. Similarly, we also get an operator $T_2:$ $V_2 \rightarrow V_1$. The operator $T_1 \oplus (-T_2) \in \mathfrak{g}_1$, i.e. is a skew-self adjoint operator on $V$ with block diagonals zero. 

Notice that the space $W$, consisting of self-adjoint operators on $V$ with block diagonal zero, is also a representation of $G^{\theta}$. This representation is also isomorphic to $V_1 \otimes V_2$,  where $\alpha$ would map to $T_1 \oplus T_2$.

The $G$-invariant functions on $\mathfrak{g}$ restrict to $G^{\theta}$-invariant functions on $\mathfrak{g}_1$. Let $T'\in \mathfrak{g}$. Since $T'$ is skew self-adjoint, the coefficients of the odd powers of the characteristic polynomial will all be zero. Suppose that the characteristic polynomial of $T'$ is $g_1(x) = f_1(x^2) = x^{2n} + b_1 x^{2n-2} + \hdots + b_{n-1}x^2 + b_n$. Because $T'$ is skew self-adjoint, $b_n = (-1)^ne^2$, where $e$ is the pfaffian of $T'$. The functions $b_1, \hdots b_{n-1}, e$ freely generate the ring of $G$-invariant functions on $\mathfrak{g}$. By Vinberg's theory, the ring of $G^{\theta}$-invariant functions on $\mathfrak{g}_1$ is freely generated by $b_1, \hdots, b_{n-1}, e$ if the characteristic of $k$ is 0 (\cite[Theorem 3.6]{Pany}). 

If the characteristic polynomial of the associated skew self-adjoint operator is as above, the characteristic polynomial of the associated self-adjoint operator $T_1 \oplus T_2$ will just be $g(x) = f(x^2) = x^{2n} + a_1 x^{2n-2} + ... a_{n-1}x^2 + a_n$ with $a_i = (-1)^i b_i$. Note that we now get $e^2 = a_n$. Henceforth, we will think of $\alpha \in V_1 \otimes V_2$ as a self-adjoint operator $T$ on $V$ with block diagonal zero. For notational reasons, we will use the symbol $\alpha$ when we want to talk of an element of $V_1 \otimes V_2$ in the abstract, and we will use the symbol $T$ when we want to think of $\alpha$ as a self-adjoint operator on $V$. We call $g$ the characteristic polynomial of $\alpha$.  

The invariants $a_i,e$ are homogenous functions, with the degree of $a_i$ being $2i$, and the degree of $e$ being $n$. Note that the sum of the degrees of the invariants is $n^2$, the dimension of $V_1\otimes V_2$. 

Define $\Inv = \Spec k[a_1,a_2, \hdots, a_{n-1}, e]$. If the characteristic of $k$ is 0, then $\Inv \simeq V_1 \otimes V_2 \sslash G^{\theta}$, the GIT quotient of $V_1 \otimes V_2$ by $G^{\theta}$. Even if $k$ has positive characteristic, there is still a $G^{\theta}$-equivariant map from $ V_1 \otimes V_2$ to $\Inv$, where the action of $G^{\theta}$ on $\Inv$ is trivial. In either case, $\pi: V_1 \otimes V_2 \rightarrow \Inv$ denote the $G^{\theta}$-equivariant map. 

\begin{Definition}
We say that an element $\alpha \in V_1 \otimes V_2$ is regular semisimple if its characteristic polynomial splits into distinct linear factors over $k^{\sep}$, i.e. if the discriminant of $g(x)$ is different from zero. 
\end{Definition}
Note that $\alpha$ will be regular semisimple if and only if the polynomial $f$ has non-zero discriminant and $f(0) \neq 0$. In terms of the map $\pi$, the regular semisimple locus in $V_1 \otimes V_2$ equals $\pi ^{-1}(\Invrs)$, where $\Invrs$ is the locus where $e$ and the discriminant of $f$ are both non-zero. 

Henceforth, we assume that $n$ is an odd integer. Let $n =2m+1$. 
\subsubsection*{Regular nilpotent orbits and Kostant sections}
For this paragraph, we assume that the characteristic of $k$ is zero. We also think of $V_1 \otimes V_2$ as being $\mathfrak{g}_1$, i.e. as skew self-adjoint operators with block diagonal zero. 

\begin{Proposition}\label{regnilp}
There are exactly two distinct $G^{\theta}(k)$-orbits of regular nilpotent elements of $\mathfrak{g}_1$. 
\end{Proposition}
\begin{proof}
Let $Z$ and $G_{\ad}$ denote the center and adjoint of $G$ respectively. Clearly, $\theta$ acts trivially on $Z$, and hence descends to an involution of $G_{\ad}$. Let $(G_{\ad})^{\theta}$ denote the subgroup of $G_{\ad}$ fixed by $\theta$. 

We have that $(G_{\ad})^{\theta}(\overline{k}) = \{g \in G(\overline{k})| \theta(g) \in Z(\overline{k}) g \}/Z(\overline{k})$. Clearly, $G^{\theta}/Z$ is a subgroup of $(G_{\ad})^{\theta}$, and the description of $\overline{k}$-points shows that the inclusion has index two. The group $G^{\theta}/Z$ is isomorphic to $\SO(V_1) \times \SO(V_2)$, and is therefore the connected component of  $(G_{\ad})^{\theta}$. Finally, it is easy to see that there always exists $g \in G(k)$ with the property $\theta(g) =-g$. It follows that $(G^{\theta}/Z)(k)$ is always an index-two subgroup of $G_{\ad}^{\theta}(k)$. 

By \cite[Lemma 2.13]{Thorne}, $G_{\ad}^{\theta}(k)$ acts simply transitively on the set of regular nilpotent elements of $\mathfrak{g}_1(k)$. Therefore, the action of $(G^{\theta}/Z)(k)$ on the set of regular nilpotent elements has two orbits, as required. 
\end{proof}
We now explicitely describe these two conjugacy classes of regular nilpotent elements. By the assumpion on $Q_1$, there exists a basis $\{f_1, \hdots \ f_{2m+1}\}$ of $V_1$, such that the Gram matrix of $Q_1$ is 
\begin{equation}\label{stdform}
B = 
\left[
\begin{array}{ccccc}
&&&&1\\
&&&1&\\
&&\iddots&&\\
&1&&&\\
1&&&&
\end{array}
\right].
\end{equation}
Similarly, there exists a basis $\{f'_1, \hdots \ f'_{2m+1} \}$ of $V_2$ such that the Gram matrix of $Q_2$ is $-B$. Let $E_1 \in \mathfrak{g}_1$ be as follows: 
$$f_1 \rightarrow f'_1, \ f'_1 \rightarrow f_2, \ \hdots \ f'_{m-1} \rightarrow f_m, \ f_m \rightarrow f'_{m+1}, \ \hdots \ f'_{2m+1}\rightarrow f_{2m+1} \rightarrow 0; \ f'_{m} \rightarrow 0.$$
Let $E_2$ be defined the same way, except with the $f_i$ and $f'_i$ swapped. Both $E_1$ and $E_2$ are regular nilpotent elements of $V_1 \otimes V_2$ (thought of as symmetric operators with block diagonal zero). It is easy to see that $E_1$ and $E_2$ are in the same $G_{\ad}^{\theta}(k)$ orbit, but are in different $G^{\theta}(k)$ orbits. 

By \cite[Lemma 2.15]{Thorne}, $E_1$ can be completed to an $\mathfrak{s}\mathfrak{l}_2$ triple $(E_1, F_1, H_1)$ with $F_1 \in \mathfrak{g}_1$ and $H_1 \in \mathfrak{g}^{\theta = 1}$ in a unique way. The same is true for $E_2$ (and we call the triple $(E_2,F_2,H_2)$). Let $\mathfrak{z}(F_1) = \{Z\in \mathfrak{g}_1: \left[Z, F_1\right] =0 \}$ (and let $\mathfrak{z}(F_2)$ be defined analogously). By \cite[Lemma 3.5]{Thorne}, the two $\kappa'_i: E_i + \mathfrak{z}(F_i) \rightarrow \Inv$ (for $i=1,\ 2$) are isomorphisms, and thus give rise to two sections $\kappa_i: \Inv \rightarrow \mathfrak{g}_1$. 

\begin{Definition}
The sections $\kappa_i$ are called Kostant sections.
\end{Definition}

\begin{Proposition}\label{kost}
Let $T' \in E_i + \mathfrak{z}(F_i)$ (where $i$ is either 1 or 2) be a regular semisimple element. Then, there exists $X \subset V_i$ a maximal isotropic subspace (for the quadratic form $Q_i$), with the property $T'^2 X \subset X^{\perp}$ ($^{\perp}$ is taken with respect to the quadratic form $Q_i$).
\end{Proposition}
\begin{proof}
It suffices to prove the result for $i = 1$. We will show that $X$ can be chosen to be $\langle f_1, f_2, \ \hdots f_m \rangle$ where, $\langle \ \rangle$ denotes the span. We have $X^{\perp} = \langle X, f_{m+1} \rangle$. Notice that the $\langle f_i \rangle$ and $\langle f'_i \rangle$ will be weight-spaces for $H_1$. Hence, we have that $F_1$ acts as follows: 
$$\langle f_1 \rangle \leftarrow \langle f'_1 \rangle, \ \langle f'_1 \rangle \leftarrow \langle f_2 \rangle, \ \hdots \ \langle f'_{m-1} \rangle \leftarrow \langle f_m \rangle, \ \langle f_m \rangle \leftarrow \langle f'_{m+1} \rangle , \ \hdots \ \langle f'_{2m+1}\rangle \leftarrow \langle f_{2m+1} \rangle;$$
$$\langle f_1, f'_{m} \rangle \rightarrow 0.$$
The kernel of $F_1 ^ {2m}$ is $\langle f_1, \hdots f_m, f'_1 \hdots f'_{m+1} \rangle$, and the kernel of $F_1^{2m+1}$ is $\langle f_1, \hdots f_{m+1}, f'_1 \hdots f'_{m+1} \rangle$. Thus, any $Z \in \mathfrak{z}(F_1)$ must preserve the above two subspaces. Therefore, 
$$\displaystyle{T' \langle f_1, \hdots f_m, \ f'_1, \hdots f'_m \rangle \subset \langle f_1, \hdots f_m, \ f'_1 \hdots f'_{m+1} \rangle},$$ 
$$\displaystyle{T' \langle f_1, \hdots f_m, \ f'_1 \hdots f'_{m+1} \rangle \subset \langle f_1, \hdots f_{m+1}, \ f'_1 \hdots f'_{m+1} \rangle}$$ and consequently
$$\displaystyle{T'^2 \langle f_1, \hdots f_m, \ f'_1, \hdots f'_m \rangle \subset  \langle f_1, \hdots f_{m+1}, \ f'_1 \hdots f'_{m+1} \rangle}.$$

The proposition follows from the fact that $T'^2$ preserves $V_1$. 
\end{proof}
Recall that we associated a self-adjoint operator $T$ to each $T' \in \mathfrak{G}_1$, and that $T$ and $T'$ satisfy the relation $T^2 = -T'^2$. Therefore, if $\alpha \in \kappa_i(\Inv)$, the corresponding (self-adjoint) $T$ satisfies the conclusions of Proposition \ref{kost}. To that end, for $k$ having characteristic different from two, we make the following definition.
\begin{Definition}
We call $\alpha \in V_1 \otimes V_2$ and the associated self adjoint operator $T$ $i$-distinguished (for $i = 1,2$) if there exists a maximal isotropic subspace (with respect to the form $Q_i$) $X \subset V_i$ such that $T^2 X \subset X^{\perp}$ ( $^{\perp}$ is taken with respect to $Q_i$).
\end{Definition}

\subsection{$G^{\theta}_0$ from $G^{\theta}$}
Let $G_0^{\theta} = \SO(V_1) \times \SO(V_2)$ be the connected component of $G^{\theta}$ containing the identity. Note that it has index $2$ in $G^{\theta}$. As $n$ is odd, $-\Id_n \in \GL(V)$ is an element of $G^{\theta}$, but not of $G_0^{\theta}$. The element $-\Id_n$  acts trivially on $V_1 \otimes V_2$, and so $V_1 \otimes V_2 \sslash G^{\theta} = V_1 \otimes V_2 \sslash G_0^{\theta}$. Further, for any field $k$, we have that $V_1 \otimes V_2 (k) / G_0^{\theta}(k) = V_1 \otimes V_2 (k) / G^{\theta}(k)$. 

Henceforth we define $G_i$ to be $\rm{SO}(V_i)$ for $i = 1,2$. For ease of notation, let $H = G_1 \times G_2$. Recall that we have identified $W$, the space of self-adjoint operators with block diagonal zero, with $V_1 \otimes V_2$ as $H$ representations. We will henceforth use the notation $W$ for $V_1 \otimes V_2$.

\section{Orbits over a separably closed field}
Recall that we have assumed that $n = 2m+1$ is an odd integer. In this section, we prove that over a separably closed field $k$, regular semisimple elements having the same invariants lie in the same $\Gij(k)$ orbit. 

\begin{Proposition}
Let $S$ and $T$ be regular semisimple elements in $W$, with block diagonal zero. Suppose that $S$ and $T$ have the same invariants. Then, there exists $g \in\Gij(k)$ such that $gSg^{-1} = T$. 
\end{Proposition}
\begin{proof}
Suppose that the common characteristic polynomial is $g(x) = f(x^2)$. If $\lambda$ is an eigenvalue, then so is -$\lambda$. Let $w_{\pm 1}$, $w_{\pm 2}$, ... $w_{\pm n}$ and $w'_{\pm 1}$, $w'_{\pm 2}$, ... $w'_{\pm n}$ be the eigenvectors of $S$ and $T$ respectively, with eigenvalues $\pm \lambda_1$ ... $\pm \lambda_n$. By considering $S^2$, we see that for any $i$, the span of the two vectors $w_{\pm i}$ intersected with $V_1$ (and $V_2$) is one-dimensional. Of course, the same reasoning applies to the span of the two vectors $w'_{\pm i}$. Note that the $w_i$ (and the $w'_i$) form a basis orthogonal for the form on $V$, since $S$ and $T$ are self-adjoint, and the eigenvalues $\pm\lambda_i$ are distinct.

Without loss of generality, assume that $w_i + w_{-i}$ (and the same with $w'$) lies in $V_1$ for every $i$. As $S$ maps $V_1$ to $V_2$, $\lambda_i(w_i - w_{-i})$ and therefore $w_i - w_{-i}$ has to lie in $V_2$. We have that $(w_i+w_{-i},w_i - w_{-i}) = 0$ where $\Bij$ is the bilinear form associated with $Q$, whence $Q(w_i) = Q(w_{-i})$. Clearly, the same happens with the $w'_{\pm i}$. By scaling $w'_{\pm i}$ appropriately, we may assume that $Q(w_{\pm i}) = Q(w'_{\pm i})$. 

Therefore, the transformation $g$ taking the $w_{\pm i}$ to the $w'_{\pm i}$ is orthogonal (the bases $w_i$ and $w'_i$ being orthogonal for the form), and conjugation by $g$ takes $T$ to $S$. Finally, the fact that the $(w_i + w_{-i})$ ({\it resp.} $(w_i - w_{-i})$) span $V_1$ ({\it resp.} $V_2$) implies that $g$ preserves $V_1$ ({\it resp.} $V_2$). Therefore, $g$ lies in the subgroup $\rm{O}(V_1) \times \rm{O}(V_2)$. Conjugating by $g$ multiplies the pfaffian by the determinant of $g$. The fact that the two operators $S$ and $T$ have the same pfaffian forces $g$ to lie in $\SL(V)$, which means that $g \in G^{\theta}$. As remarked at the end of the previous section, we can choose $g$ be an element of $\Gij(k)$. 
\end{proof}

The corollary below follows directly from the above proposition. 
\begin{Corollary}\label{gorb}
For $k$ a separably closed field, two regular semisimple elements $\alpha_1$ and $\alpha_2$ lie in the same $\Gij (k)$-orbit if and only if they have the same invariants.
\end{Corollary}

\section{Orbits over an arbitrary field}
In this section, we first demonstrate the existence of rational orbits with a given set of invariants, and describe using the theory developed in \cite{AIT} how a geometric orbit decomposes into rational orbits. 
\subsection{Existence of orbits with a given set of invariants}
Fix a set of invariants $c=(a_1, a_2, a_3, \hdots a_{n-1},e) \in \Invrs(k)$. Recall that we have associated to $c$, two polynomials $f(x) = x^n + a_1 x^{n-1} + \hdots a_{n-1}x + e^2$, and $g(x) = f(x^2)$. The existence of the Kostant sections prove that when the characteristic of $k$ is zero, the set of $k$-rational orbits with invariants $c$ is non-empty, i.e. $\pi^{-1}(c)(\overline{k})$ contains a $k$-rational point. We construct an explicit $T \in \Vij(k)$ with invariants $c$. The construction holds for $k$ having characteristic different from two. 

Let $L$ = $k[x]/(f(x))$ and let $M$ = $k[x]/(g(x))$. There is an embedding of $k$-algebras $L \hookrightarrow M$ such that $x \mapsto x^2$. Let $\sigma$ be the non-trivial automprphism of $M$ which leaves $L$ fixed. Let $\beta$ and $\gamma$ be the images of $x$ in $M$ and $L$ respectively. By definition, $\beta$ = $\gamma ^2$ and $\sigma$ sends $\beta$ to $- \beta$. Define a symmetric bilinear form $\B$ on $M$ by setting 
$$(\lambda , \mu) = \Tr_{M/k}(f'(\gamma) \lambda \mu). $$ 
For $\mu \in L$, we have that $\Tr_{M/k}\beta \mu = 0$. Therefore $\B$ breaks up into a direct sum of split bilinear forms on $L$ and $L \beta$. An easy computation shows that these spaces have discriminants $1$ and $(-1)^n=(-1)$ respectively. Isometrically identifying $L$ and $L\beta$ with $V_1$ and $V_2$ respectively, the self-adjoint operator $T_{\beta}$ on $M$ (given by multiplication by $\beta$) pulls back to a self-adjoint operator $T$ on $V$. Clearly, $T$ is self-adjoint, maps $V_1$ to $V_2$ and $V_2$ to $V_1$, and therefore corresponds to an orbit of $\alpha_1 \in W(k)$ of our representation. The invariants of $T$ are $c$ by construction, up to the sign of the pfaffian. To obtain an operator with the same invariants except for the sign of the phaffian being reversed, replace $\beta$ by $-\beta$. 

For $c \in \Invrs(k)$, let $\alpha_1(c) \in \Vij(k)$ denote the orbit just constructed. 

\begin{Proposition} \label{stabis}
The stabilizer $\Stab_c$ of $\alpha_1(c)$ is isomorphic to the kernel of the norm map $\Res_{L/k}(\mu _2) \rightarrow \ \mu_2$.
\end{Proposition}
\begin{proof}
The stabilizer $\Stab_c$ of this orbit is the space of orthogonal linear transformations on $M$ which preserve $L$ and $L\beta$, commute with $T_\beta$, have determinant $1$ when restricted to $L$ and $L\beta$.

As the orbit is regular semisimple, the centralizer of $T_\beta$ in $\GL(M)$ ($M$ thought of as a $k$-vector space) is $M\x$ (acting on $M$ by mutiplication). That $\Stab_c$ is a subgroup of $\SO(M)$ implies that only elements $\lambda \in M\x$ of the form $\lambda^2 = 1$ are allowed. We have $M\x \cap \GL(V_1) \times \GL(V_2) = L\x$, therefore $\lambda$ to belong to $L\x$. Finally, the determinant condition forces $N_{L/k}(\lambda) = 1$. The proposition follows. 
\end{proof}

\begin{Corollary}\label{bigstabred}
We have $\Stab_c(k) \neq \{1\}$ precisely when the polynomial $f$ is not irreducible. 
\end{Corollary}
\begin{proof}
The etale algebra $L$ is a field precisely when $f$ is irreducible. Applying Proposition \ref{stabis} finishes the proof. 
\end{proof}
Note that the intersection of $L\x$ (and therefore $\Stab_c$) with $\GL(V_1) \times \{1\}$, and with $\{1\} \times \GL(V_2)$ is just the identity. This implies that either projection restricted to $L\x$ (and therefore to $\Stab_c$) is an isomorphism onto its image.

\begin{Corollary}
For $c\in \Invrs(k)$, $\Stab_c$ is isomorphic to the stabilizer of any other $\alpha \in \Vij(k)$ with the same invariants.
\end{Corollary}
\begin{proof}
Let $T$ be the operator corresponding to $\alpha$. By Proposition \ref{gorb}, there exists $g \in \Gij(k^{\sep})$ such that $gT$ is the operator associated to $\alpha_1(c)$, which for ease we denote by $S$. Conjugation by $g$ provides an isomorphism between the stabilizers of $S$ and $T$, apriori defined only over $k^{\sep}$. Clearly, conjugation by $\sigma(g)$ is the isomorphism between the stabilizers obtained by applying $\sigma$ to the previous isomorphism, where $\sigma \in \Gal(k^{\sep}/k)$.

We claim that the two maps are the same. Indeed, $S$ and $T$ are $k$-rational, so that $\sigma(g)T = gT$, thereby forcing $g^{-1}\sigma(g) \in \Stab_c$. Therefore, the two maps differ by conjugation by an element of $\Stab_c$. That $\Stab_c$ is abelian forces the two maps to be the same. Therefore, the isomorphism between the two stabilizers is defined over $k$, proving the result. 
\end{proof}

\subsection{Distinguished orbits}
\begin{Proposition}
Fix $c\in \Invrs(k)$. Then $\Gij(k)$ acts simply transitively on the set of pairs $(T,X)$ where $T \in \Vij(k)$ and $\pi(T) = c$, and $X\subset V_1$ is a maximal isotropic subspace, with the property that $T^2X \subset X^{\perp}$. 
\end{Proposition}
\begin{proof}
We first show that the stabilizer of such pairs is just the identity. Let $g = (g_1,g_2) \in \SO(V_1) \times \SO(V_2)(k)$ be an element in the stabilizer. Then $g_1$, thought of as an element of $\SO(V_1)$ commutes with the self adjoint operator $T^2$ restricted to $V_1$ and also preserves the subspace $X$. By \cite{BG}, $g_1 = \Id_n$. As remarked above, this forces $g_2 = \Id_n$, as required. Because this stabilizer is trivial, an easy descent argument shows that it suffices to prove the statement over $k^{\sep}$. 

By Proposition \ref{gorb}, it suffices to prove that $(T,X')$ and $(T,X)$ are in the same orbit, where $X'$ is a subspace with the same properties as $X$. By \cite[Proposition 6]{BG}, there exists $g_1 \in \SO(V_1)$ which commutes with $T^2$, such that $g_1X' = X$. It suffices to demonstrate the existence of $g_2 \in \SO(V_2)$ such that $g = (g_1,g_2) \in \Stab_c$. By \cite{BG}, the centralizer $\Stab_{V_1}$ of $T^2$ in $\SO(V_1)$ is an abelian $2$-group of order $2^{n-1}$, which is the same as the order of $\Stab_c$. Therefore, the projection map (having trivial kernel) from $\Stab_c$ to $\Stab_{V_1}$ must be a bijection, whence we deduce the existence of the required $g_2$. 
\end{proof}

It is easy to see that $\alpha_1 \in \Vij(k)$ is 1-distinguished. Similarly, in the next subsection, we will explicitly construct a corresponding $\alpha_2$ which is 2-distinguished. Over a separably closed field, these two operators will lie in the same $\Gij$-orbit. 

\subsection{The remaining orbits}
We again fix $c \in \Invrs(k)$. Let $W_c$ denote the fiber of $\pi$ over $c$. Proposition \ref{gorb} can be rephrased as stating that $\Gij(k^{\sep})$ acts transitively on $W_c(k^{\sep})$. Once one particular $\Gij(k)$-orbit is fixed,  by \cite[Proposition 1]{AIT}, the set $W_c(k)/\Gij(k)$ is in bijection with the kernel of a map of pointed sets $\delta:\ H^1(k,\Stab_c) \rightarrow H^1(k,\Gij)$ (the notation used in \cite{AIT} is $\gamma$, not $\delta$). We use $\alpha_1(v)$ (the 1-distinguished orbit) as our fixed orbit, and explicitly describe the map $\delta$.

The Kummer exact sequence gives that $H^1(k,\Stab_c)$ = $(L\x/L\xt)_{N=1}$. To an element $\nu$ in $L\x_{N=1}$, we associate the orthogonal space $M$ with the bilinear form $\BB_{\nu}$, with $\langle \lambda , \mu \rangle_{\nu} = \Tr_{M/k}(f'(\gamma) \nu \lambda \mu)$. This orthogonal space corresponds to an element in $H^1(k,\SO(V))$. Clearly, the new form breaks up into a direct sum of forms on $L$ and $L\beta$, so our cocycle actually lies in $H^1(k,\Gij)$ (recall that $\Gij = G_1 \times G_2$). Let $\B_{\nu}$ denote the bilinear form on $L$ given by $(\lambda, \mu)_{\nu} = \Tr_{L/k}(f'(\gamma) \nu \lambda \mu)$. It is easy to see that $\BB_{\nu} = \B_{\nu} \oplus \B_{\text{-}\nu\gamma}$. The map $\delta$ maps the class $\nu$ to the element of $H^1(k,\Gij)$ just described (\cite[Lemma 3]{AIT}). The class of $\nu$ will be in the kernel precisely when both the spaces $L$ and $L\beta$ are split. Note that this is same as saying that the forms on $L$ given by $\nu$ and $\nu \gamma$ are both split. 

It is easy to see that the element $\nu = (-\gamma) \in H^1(k,\Stab_c)$ does lie in the kernel of $\delta$. Therefore, $(-\gamma)$ gives rise to $\alpha_2 \in W_c(k)$, and its class in $H^1(k,\Stab_c)$ gives a $k$-rational orbit in $W_c(k)/\Gij(k)$. In fact, the orbit of $\alpha_2$ is two-distinguished. Summarizing, we have the 
\begin{Proposition}
For each $c \in \Invrs(k)$, the $\Gij(k)$-orbit of $\alpha_1(v) \in \Vij(k)$ is 1-distinguished. The stabilizer of $\alpha_1$,  $\Stab_c$ is isomorphic to $\Res_{L/k}(\mu _2)_{N=1}$. All the other $k$-rational orbits orbits with the same invariants have the same stabilizer, lie in the $\Gij (k^{\sep})$-orbit of $\alpha_1$ and correspond bijectively to the non-identity classes in the kernel of $\delta :  H^1(k,\Stab_c) \rightarrow H^1(k,\Gij)$. The class of $\nu \in L\x_{N=1}$ in $H^1(k,\Stab_c)$ corresponds to the space $M$ along with the operator $T_{\beta}$ and the bilinear form $\BB_{\nu}$. The class of $(-\gamma)$ is the 2-distinguished orbit, corresponding to the $\Gij(k)$ orbit of $\alpha_2$.
\end{Proposition}
An immediate corollary is 
\begin{Corollary}\label{gammasquare}
Fix $c \in \Invrs(k)$. The two distinguished orbits lie in the same $\Gij(k)$-orbit if and only if $(-\gamma)$ is a perfect square in $L\x$.
\end{Corollary}

\section{Connection with hyperelliptic curves}
In this section, we associate hyperelliptic curves and some torsors to rational orbits of our representation. Recall that we have assumed that $n = 2m+1$ is odd. Given $c \in \Invrs(k)$, we associate the curves $C_{1,c}, C_{2,c}$ given by $y^2 = f(x)$ and $y^2 = xf(x)$. As usual, $f(x) = x^n + a_1 x^{n-1} + \hdots a_{n-1}x + e^2$, where $c = (a_1, \hdots a_{n-1},e)$. 

The curves have marked points (rational over $k$): $C_{1,c}$ has a rational Weirstrass point which we call $\infty_1$, which lies above the point at infinity in $\mathbb{P}^1$. The points $(0,\pm e)$ are also $k$-rational, and are conjugate for the hyperelliptic involution on $C_{1,c}$. We call these points $P_1,P'_1$. Similarly, the point of $C_{2,c}$ above $0$ is a rational Weirstrass point, which we call $P_2$. There is also a pair of $k$-rational points above the point at infinity, $\infty_2,\infty'_2$, which are conjugate for the hyperelliptic involution on $C_{2,c}$. Let $J_{i,c}$ be the Jacobians of $C_{i,c}$. Note that $J_{1,c}[2] \cong J_{2,c}[2] \cong \Res_{L/k}(\mu _2)_{N=1}$.

\subsection{Pencils of quadrics}
Suppose that $\alpha \in W^{\rs}_c(k)$, i.e. has invariants $c$. The 2-torsions of $J_{1,c}$ and $J_{2,c}$ are related to the stabilizer of $\alpha$ as follows:  

\begin{Proposition}
The stabilizer $\Stab_c$ is isomorphic (as group schemes over $k$) to $J_{i,c}[2]$.
\end{Proposition}

Recall that $\Gij(k)$-orbit of $\alpha_1(c)$ is 1-distinguished, and that the map $\delta$ (based at the orbit of $\alpha_1$) described in the previous section gives a map from $W_c(k)/\Gij(k) \rightarrow H^1(k,\Stab_c)$. The inclusion $A[2] \hookrightarrow A$ for any group scheme over $k$ gives the natural map $H^1(k, A[2]) \rightarrow H^1(k,A)$. Therefore, for $i = 1,2$, we have maps $ W_c(k)/\Gij(k) \rightarrow H^1(k,J_{i,c})$, by identifying $\Stab_c$ with $J_i[2]$. We recall the theory developed in \cite{Jerrythesis}, (also see \cite{AIT}, \cite{Jerry}) which describes these maps. 

Recall that the vector spaces $V_1$ and $V_2$ are equipped with quadratic forms $Q_1$ and $Q_2$. Let $B_1$ and $B_2$ denote the associated bilinear forms. Let $\alpha \in W_c(k)$, and let $T = T_1 \oplus T_2$ be the associated self adjoint matrix with block diagonal zero. Define $B_{1,T^2}(v_1,w_1) = B_1(v_1,T^2w_1)$ for $v_1, w_1\in V_1$. Note that $B_{1,T^2}$ is $B_2$ pulled back from $V_2$ by $T_1$. Denote by $Q_1$ and $Q_{1,T^2}$ the corresponding quadratic forms on $V_1$. 

Define $P^1_{\alpha}$ to be the pencil of quadrics on the space $\mathbb{P}(V_1 \oplus k)$ spanned by $Q'_1$ and $Q'_{1,T^2}$, where $Q'_1((v_1, \lambda)) = Q_1(v_1)$, and $ Q'_{1,T^2}((v_1,\lambda)) = Q_{1,T^2}(v_1,v_1) + \lambda^2$. Define $F^1_{\alpha}$ to be the Fano variety of the base locus of $P^1_{\alpha}$. The theory developed in \cite{Jerrythesis} demonstrates $F^1_{\alpha}$ as being a torsor for $J_1$. In fact, if the orbit of $\alpha$ corresponds to $\nu \in H^1(k,\Stab_c)$, then it is proved in \cite[Corollary 2.23]{Jerrythesis} that $F^1_{\alpha}$ is the image of $\nu$ in $H^1(k,J_{1,c})$. 

The same construction with the order of $Q_1$ and $Q_{1,T^2}$ reversed gives a pencil of quadrics $P^2_{\alpha}$. The asociated Fano variety $F^2_{\alpha}$ gives a torsor of the Jacobian of the curve $y^2 = x^n + \frac{a_{n-1}}{e^2}x^{n-1} + \hdots + \frac{a_1}{e^2} + \frac{1}{e^2}$. But this curve is isomorphic to $C_2$ via the isomorphism $(x,y) \mapsto (1/x,ey/x^{\frac{n-1}{2}})$. Again, $F^2_{\alpha}$ is the image of $\nu$ in $H^1(k,J_2)$, where $\Stab_c$ is identified with $J_{2,c}[2]$. 

For $i=1,2$, $\delta(\alpha)$ will be in the image of $J_{i,c}/2J_{i,c}$ precisely when $F^i_{\alpha}$ has a $k$-rational point. Also, changing only the sign of the pfaffian but leaving the other invariants fixed doesn't change $P^1_{\alpha}$ or $P^2_{\alpha}$. This is because, changing the pfaffian is the same is replacing $T$ by $-T$, and this doesn't change $T^2$. Therefore the 2-cover that we get stay the same. We will ignore the sign of the pfaffian while associating pencils to rational orbits. 

\subsection{Soluble orbits}
Any element of $J_{i,c}(k)$ can be mapped to $H^1(J_{i,c}[2])$, and through the identification with $H^1(k,\Stab_c)$, to $H^1(k,\Gij)$. We have also identified the set of $\Gij(k)$-orbits with invariants $c$ with a subset of $H^1(k,\Stab_c)$, under which the 1-distinguished orbit corresponds to the trivial element of $H^1(k,\Stab_c)$.
\begin{Proposition}\label{twodist}
The class of the 2-distinguished orbit in $H^1(k,\Stab_c)$ is in the image of $J_{i,c}$ for both $i$. 
\end{Proposition}
\begin{proof}
Stoll in \cite{Stoll} explicitly computes the 2-descent map from $J_{i,c}(k)$ to $H^1(k,J_{i,c}[2])$. With this in hand, it is easy to see that $(-\gamma)$, the class of the 2-distinguished orbit, is the image of $P_1 -\infty_1 \in J_{1,c}(k)$, and also the image of $P_2 - \infty_2 \in J_{2,c}(k)$. 
\end{proof}
Having seen that the marked points of both curves give rise to the 2-distinguished orbit, we now prove that the composite maps from $J_{i,c}(k)$ to $H^1(k,\Gij)$ are trivial, which tells us rational points in either Jacobian give us rational orbits. 
\begin{Proposition}\label{soluble}
Let $\nu$ be an element of $H^1(k,\Stab_c)$, which lies in one of the subgroups $J_{i,c}/2J_{i,c}$. Then $\nu$ lies in the kernel of $\delta$.
\end{Proposition}
\begin{proof}
The element $\delta(\nu) \in H^1(G_1 \times G_2)$ corresponds to the quadratic spaces isomorphic to $L$ with forms $\B_{\nu}$, and $\B_{-\nu \gamma}$. The element $\delta(\nu)$ is trivial precisely when $\Bij_{\nu}$ and $\Bij_{-\nu\gamma}$ are both split. 

By \cite[Proposition 6]{AIT} (see \cite[Theorem 4.6]{newJack} for proof which doesn't use pencils of quadrics) $\B_{\lambda}$ is split for any $\lambda \in J_{1,c}(k)/2J_{1,c}(k) \subset H^1(k,\Stab_c)$. If $\nu \in J_{1,c}(k)/2J_{1,c}(k)$, we apply Proposition \ref{twodist} to conclude that the same holds for $-\nu\gamma$. Therefore, the required spaces are split if $\nu$ lies in the subgroup $J_{1,c}(k)/2J_{1,c}(k)$. 

Exactly the same argument works if $\nu$ is in the image of $J_{2,c} /2J_{2,c}$ - notice that $C_{2,c}$ is isomorphic to the curve given by the Weirstrass equation $y^2 = x^n + \frac{a_{n-1}}{e^2}x^n-1 + \hdots + \frac{a_1}{e^2} + \frac{1}{e^2}$, and we apply the same result of \cite{AIT} (or \cite{newJack}) and Proposition \ref{twodist} to finish the proof of the result.  
\end{proof}

\begin{Definition}
Suppose that an orbit under $\Gij(k)$ corresponds to an $\nu \in H^1(\Stab_c)$ in the image of $J_{i,c}(k)/2J_{i,c}(k)$, for $i = 1,2$. We then say that the orbit is $i$-soluble. If an orbit lies in the image of both $J_{1,c}(k)$ and $J_{2,c}(k)$, we say that the orbit is $(1,2)$-soluble. 
\end{Definition}
Note that an orbit is (1,2)-soluble if and only if it is both 1-soluble and 2-soluble. Further, there is a geometric description of when an orbit is $i$-soluble, or (1,2)-soluble. Indeed, as mentioned above, the orbit of $\alpha$ is $i$-soluble if the corresponding Fano variety $F^i_{\alpha}$ has a $k-$rational point. The distinguished orbits are $i$-soluble for both $i = 1$ and $2$. 

We will not work with $2$-soluble orbits, and have defined what they are only for the sake of completion. 

\section{Orbits over arithmetic bases}
Suppose now that $k$ is a number field, and that $c \in \Invrs(k)$. Define $\Sel_{(1,2)}(c)$ to be the intersction of the $\Sel_2(J_{i,c})$ (for $i=1,2$) in $H^1(k,\Stab_c)$. 

We show that all elements in $\Sel_2(J_{i,c}) \subset H^1(k, J_i[2])$ give rise to orbits. 
\begin{Proposition}
Let $\nu \in \Sel_2(J_{i,c})$ for $i=1$ or $2$. Then $\delta(\nu)$ is the trivial element in $H^1(k,\Gij)$. 
\end{Proposition}
\begin{proof}
The exact same proof as in \cite{AIT} works. We prove the result for $k = \Q$, for the same proof applies in general. We need to show that the bilinear forms $\B_{\nu}$ and $\B_{\text{-}\nu\gamma}$ are split. Because $\nu \in \Sel_2(J_{i,c})$, Proposition \ref{soluble} tells us that $B_{\nu} \otimes \Q_p$ and $B_{\nu} \otimes \R$ are split (for all $p$). The same is true for $\B_{\text{-}\nu\gamma}$. Therefore, by the Hasse-Minkowski Theorem, the two forms must be split over $\Q$, as required. 
\end{proof}

We say that the $\Gij(k)$-orbit of $\alpha \in \Vij(k)$ is 1-soluble (or locally (1,2)-soluble) if $\alpha \in \Vij(k_{\nu})$ is soluble for every place $\nu$ of $k$. We henceforth work predominantly over the bases $\Z$ and $\Q$, and their completions. The main goal in this section is to prove that rational $\Gij(\Q)$-orbits on $\Vij(\Q)$ which are locally 1-soluble, and whose invariants are integral have representatives in $\Vij(\Z)$. 

To that end, let $D_1$ and $D_2$ be self-dual $\mathbb{Z}$-lattices inside $V_1$ and $V_2$ respectively. When we work with the rings $\Z_p$, we will (for sake of brevity) use the same notation $D_i$ to denote the completions of the lattices inside $V_i \otimes \Q_p$. Further, we will use the notation $\Vij(\Z)$ \textit{(resp. $V_1(\Z), V_2(\Z))$}for $D_1 \otimes D_2$ \textit{(resp. $D_1, D_2$)}. The same holds for $\Z_p$. 

Recall that there exist bases of $V_1(\Z)$ and $V_2(\Z)$, with respect to which the Gram matricies are $\pm B$, with $B$ as in Equation \eqref{stdform}.

The group $\Gij$ is defined over $\Z$, and is a reductive group scheme over $\Spec \Z[1/2]$. Given $\alpha \in \Vij(R)$, define $\Stab_{\alpha}(R) = \{g \in \Gij(R) \vert g\cdot\alpha = \alpha \}$, where $R$ stands for $\Z$ or $\Z_p$. Clearly, $\Stab_{\alpha}(R) \subset \Stab_{\alpha}(R \otimes_{\Z} \Q)$.

\begin{Definition}
Let $\pi: \Vij \rightarrow \Inv$ be the map described in \S 2. 
\begin{enumerate}
\item Define $\Inv(\Z) \subset \Inv(\Q)$ to be $\pi(\Vij(\Z))$. 
\item Define $\Invrs(\Z) \subset \Inv(\Z)$ to be $\pi(\Vij(\Z)\cap \Vij^{\rs}(\Q))$. 
\item Define $\Inv(\Z_p) \subset \Inv(\Q_p)$ to be $\pi(\Vij(\Z_p))$. 
\item Define $\Invrs(\Z_p) \subset \Inv(\Z_p)$ to be $\pi(\Vij(\Z_p)\cap \Vij^{\rs}(\Q_p))$.
\end{enumerate}
\end{Definition}
Notice that there is a reduction map from $\Inv(\Z)$ to $\Inv(\F_p)$ for $p>2$. An element of $\Invrs(\Z)$ maps to an element of $\Invrs(\F_p)$ exactly when $p$ doesn't divide the discriminant of $g(x)$, i.e. $p$ doesn't divide $e$ and $p$ doesn't divide the discriminant of $f(x)$. 

We have already seen that selmer group elements always give rise to $\Gij(\Q)$-orbits -- the locally soluble ones. The main theorem of this section is the following: 
\begin{Theorem}\label{intorb}
Suppose that $c = (a_1, \hdots a_{n-1},e)\in \Invrs(\Z)$, such that $2^{4i} \vert a_i$, and $2^{2n} \vert e$. Then every $\Gij(\Q)$-orbit which has invariants $c$ and is locally 1-soluble, has an integral representative. 
\end{Theorem}
The local versions of Theorem \ref{intorb} are:
\begin{Theorem}\label{psolintorb}
Let $c = (a_1,\hdots a_{n-1},e) \in \Invrs(\Z_p)$, where $p \neq 2$. Then every $\Gij(\Q_p)$-orbit which has invariants $c$ and is $1$-soluble, has an integral representative. 
\end{Theorem}
\begin{Proposition}\label{solintorb2}
Let $c = (a_1, \hdots a_{n-1}, e) \in \Invrs(\Z_2)$, such that $2^{4i} \vert a_i$, and $2^{2n} \vert e$. Then, every soluble $\Q_2$- orbit with invariants $c$ has an integral representative. 
\end{Proposition}
We spend the bulk of this section proving Theorem \ref{psolintorb} and Proposition \ref{solintorb2}. We also describe how $\Gij$-orbits behave over arithmetic fields. 

\subsection{Finite fields of odd characteristic}
For this subsection, let $k = \F_q$, a finite field with $q$ elements where $q$ is odd. Lang's Theorem implies $H^1(k,\Gij)$ is trivial. Therefore, for $c \in \Invrs(\F_q)$, the number of $\F_q$-orbits with invariants $c$ equals the cardinality of $H^1(k,\Stab_c)$. 

Similarly, $H^1(k,J_{i,c})$ also equals zero. Therefore, the map from $J_{i,c}(k)/2J_{i,c}(k)$ to $H^1(k, J_{i,c}[2])$ is an isomorphism. Hence, every $\Gij(k)$-orbit is soluble when $k$ is a finite field. 

\subsection{The $p$-adics for $p \neq 2$}
Let $k=\Q_p$, where $p \neq 2$. Let $c \in \Invrs(\Q_p)$. We have the following well known result about soluble orbits (for instance, see \cite{Stoll}):
\begin{Proposition}\label{pstoll}
Let $J$ be the Jacobian of a hyperelliptic curve over $\Q_p$. The quantity $b_{p}=\displaystyle{\frac{J(\Q_{p})/2J(\Q_{p})}{J[2](\Q_{p})}}$ is 1, independent of $J$.
\end{Proposition}
We now give an ideal-theoretic description of integral orbits. We first cite a result of \cite{BG} which we will need:
\begin{Lemma}\cite[Lemma 15]{BG}\label{BGintorb}
Let $I$ be a $\Z_p$-module of rank $n$ equipped with a symmetric binear form $I \times I \rightarrow \Z_p$. Suppose that $I \otimes_{\Z_p} \Q_p$ is split. If the discriminant of $I$ is $1$, then $I$ is isometric to $D_1$, and if the discriminant is $-1$, then $I$ is isometric to $D_2$. 
\end{Lemma}
The condition that the form be split is unneccessary. However, we have added because then the result holds even for $\Z_2$. 

Using this result, we will give an ideal-theoretic description of $Z_p$-orbits of $\Gij(\Z_p)$. An element $\alpha$ in $D_1 \otimes D_2$ corresponds to an operator $T$ from $D_1 \oplus D_2$ to $D_2 \oplus D_1$, just as in the case of fields. Let $f$ and $g$ be as above. We identify $V_1 \oplus V_2$ with $M$ (the etale $\Q_p$-algebra $\Q_p[x]/(f(x^2))$ as defined in \S 4). Since $T$ is integral, the lattice $D_1 \oplus D_2$, is realised as a $\Z_p[x]/(f(x^2))$ submodule of $M$, which we call a fractional ideal $J$ of $\Z_p[x]/(f(x^2))$. We can in fact say more - that $T^2$ stabilizes each of the $D_i$ forces the fractional ideal $J$ to be of the form $J = I_1 + \beta I_2$, $I_1$ and $I_2$ being fractional ideals of the ring $\Z_p[x]/(f(x))$. Here, we identify $I_i$ with $D_i$. Since $T$ maps $D_1$ to $D_2$, we have $I_1 \subset I_2$. Similarly, we must have $\gamma I_2 \subset I_1$. 

The bilinear form on $M$ is of the form $\Bij_{\nu}$ for $\nu \in L^{\times}_{N=1}$. The conditions that the lattices $D_i$ are self dual translate to $\nu I_1^2 \subset \mathbb{Z}_p [x] /f(x)$, and $N(I_1)^2 = N(\nu)^{-1}$,  $\nu \gamma I_2^2 \subset \mathbb{Z}_p [x] /f(x)$, and $N(I_2)^2 = N(-\gamma \nu)^{-1}$. In sum, we have just proved the following proposition.
\begin{Proposition}\label{classint}
Assume that $f,e$ with $f(x) = x^n + a_1 x^{n-1} + ... + a_{n-1}x + a_0,$ $e^2 = a_0 \neq 0$, is a polynomial with coefficients in $\Z_p$ with nonzero discriminant. Then the integral orbits of $\Gij(\Z_p)$ on $D_1 \otimes D_2$ with invariants $a_i,e$ correspond to equivalence classes of triples $(I_1, I_2,\nu)$. Further, $\nu \in L^{\times}/L^{\times2}_{N=1}$, and the $I_i$ are fractional ideals for the order $R = \Z[x]/(f(x))$ satisfying $I_1 \subset I_2 \subset \gamma^{-1}I_1$. The element $\nu$ has the properties that the bilinear forms $\Bij_{\nu}$, $\Bij_{-\nu\gamma}$  are split forms over $\Q_p$, that $\nu I_1^2 \subset \mathbb{Z_p} [x] /f(x)$,  $\nu \gamma I_2^2 \subset \mathbb{Z_p} [x] /f(x)$, and that $N(I_1)^2 = N(\nu)^{-1}$, $N(I_2)^2 = N(-\gamma \nu)^{-1}$. The triple $(I_1, I_2,\nu)$ is equivalent to $(I'_1, I'_2,\nu')$ if $I_i = \lambda I'_i$ and $\nu' = \lambda^2 \nu$, for some $\lambda \in L^{\times}$. The integral orbit corresponding to the triple $(I_1,I_2,\nu)$ maps to the rational orbit of $\Gij (\Q_p)$ on $\Vij (\Q_p)$ corresponding to the class of $\nu$ in $(L^{\times}/L^{\times2})_{N=1}$.
\end{Proposition}
The condition that the forms $\Bij_{\nu}$ and $\Bij_{-\gamma\nu}$ are split is unneccesary. However, we have added it, because it makes the result hold even for $\Z_2$.

In case $R=\Z_p[x]/(f(x))$ happens to be the maximal order, we see that the integral orbits are in bijection with $(R^{\times}/R^{\times2})_{N\equiv 1}$. This is always true when $p$ does not divide $\Disc(f(x))$ (equivalently, when $J_1$ has good reduction). In this case, the $1$-soluble $\Gij(\Q_p)$ orbits have a particularly nice form:
\begin{Proposition}\label{vtemp}
If $p$ does not divide the discriminant of $f(x)$, then the integral orbits with invariants $c$ are in bijection with 1-soluble orbits. 
\end{Proposition}
\begin{proof}
The argument immediately preceding \cite[Corollary 18]{BG} applies verbatim.
\end{proof}
If $c$ modulo $p$ is actually regular semisimple, then both the $J_i$ have good reduction. In this case, we have the following strengthening of Proposition \ref{vtemp}: 
\begin{Proposition}
If $c$ modulo $p$ is regular semisimple, then the image of $J_1(\Q_p)$ in $H^1(k,\Stab_c)$ is the same as the image of $\J_2(\Q_p)$ (i.e. the $1$-soluble orbits are the same as the $(1,2)$-soluble orbits). Further, the  integral orbits with invariants $c$ are in bijection with 1-soluble orbits (and $(1,2)$-soluble orbits). 
\end{Proposition}
We omit the proof, as it mimics that of Proposition \ref{vtemp}. We now turn to the proof of Theorem \ref{psolintorb}. 
\begin{proof}[Proof of Theorem \ref{psolintorb}]
Suppose that $\nu \in (L^{\times}/L^{\times2})_{N=1}$ corresponds to a 1-soluble $\Gij(\Q_p)$-orbit. As both $\nu$ and $-\nu\gamma$ are 1-soluble, by \cite[Proposition 19]{BG}, there exist ideals $I_1$ and $I_2$ which satisfy all the properties of the previous proposition, except perhaps for the conditions $I_1 \subset I_2 \subset \gamma^{-1}I_1$. We will work with $I_1$ and deduce the existence of $I_2$ from it, where $I_2$ satisfies the inclusion conditions. 

A fractional ideal of $R$ corresponds to a full-rank $Z_p$ module contained in $L$, which is stable by multiplication by $\gamma$. Clearly, any lattice $\Lambda$, with $I_1 \subset \Lambda \subset \gamma ^{-1} I_1$ is stable under multiplcation by $\gamma$, and hence must be a fractional ideal. Therefore, we just need to find a $\Lambda$ satisfying the above inclusion relations, and which is self dual for the bilinear form $\Bij_{\nu\gamma}$. 

Note that by choice $I_1$ is self dual for the bilinear form $\Bij_{\nu}$, therefore $I_1$ and $\gamma^{-1} I_1$ are dual to each other for the form $\Bij_{\nu \gamma}$. By a result of Cassels \cite[Lemma 3.4]{Cassels}, there exists a $\mathbb{Z}_p$ basis $(f_i)$ of $I_1$ such that the Gram matrix for $\Bij_{\nu \gamma}$ is 
\begin{equation*}
\left(
\begin{array}{cccc}
u_1p^{b_1}&&& \\
&u_2p^{b_2} &&\\
&&\ddots &\\
&&&u_np^{b_n}
\end{array}
\right)
\end{equation*} where the $u_i$ are units in $\Z_p$. 

By replacing $f_i$ by $p^{-[b_i/2]}f_i$, we may assume that the $b_i$ are all $1$ or $0$. It is clear that the lattice $\Lambda$ spanned by $f_i$ is still sandwiched between $I_1$ and $\gamma^{-1}I_1$. Suppose that $\Lambda = \Lambda_0 \oplus  \Lambda_1$ ,where $\Lambda_j$ is the $\Z_p$-span of those $f_i$ with $b_i = j$ ($j = 0,1$). Since the discriminant of $B_2$ is 1 modulo squares (and therefore has to have even $p$-adic valuation), the dimension of $\Lambda_1$ is forced to be even (hence, the dimension of $\Lambda_0$ is odd). Let the dimension of $\Lambda_1$ be $2a$. Without any loss of generality, assume that $\Lambda_1$ is spanned by $f_1, \hdots f_{2a}$.

In particular, $\Lambda_0$ is a quadratic space of odd dimension, with the form being non-degenerate modulo $p$. Therefore, $\Lambda_0 \otimes \Q_p$ is a split quadratic space. Suppose that $\Lambda_1 \otimes \Q_p$ were also a split space. Then, by choosing a different basis $f'_1 \hdots f'_{2a}$ of $\Lambda_1$, we may assume that the Gram matrix of $\frac{1}{p} B_2$ restricted to $\Lambda_1$ is $B$. 

By replacing $\Lambda_1$ by the span of $f'_1/p, \hdots, f'_a/p, f'_{a+1},\hdots,f'_{2a} $, we see that $\Lambda = \Lambda_0 \oplus \Lambda_1$ is now self dual for $B_2$, and that $I_1 \subset \Lambda \subset \gamma^{-1}I_1$. 

Therefore, it remains to show that $\Lambda_1$ is split i.e. the Hasse invariant and the discriminant are both 1. As $(\Lambda_0 \oplus \Lambda_1) \otimes \Q_p$ with $B_2$ is also split, this means that the Hasse invariant $B_2$ is 1. Computing the Hasse invariant using the Gram matrix of $B_2$ in the basis $f_i$, it is clear that it equals the Hasse invariant of $B_2$ restricted to $\Lambda_1$, which therefore has to have  Hasse invariant 1. In terms of the $u_i$ and $m$, the Hasse invariant equals $\displaystyle (-1)^{\epsilon(p)m}\prod_{i=1}^{2a} \left(\frac{u_i}{p}\right)$ where $\epsilon(p) = (p-1)/2$. The discriminant of $\Lambda_1$ equals $(-1)^m\displaystyle{\prod_{i=1}^{2a}}u_i$ modulo squares. By definition, $u_i $ modulo squares in $\Z_p^{\times}$ equals $\displaystyle \left(\frac{u_i}{p}\right)$ (this is after identifying $\Z_p\x$ modulo squares with the group $\{\pm1\}$). Further, modulo squares $\displaystyle (-1)^m =   (-1)^{\epsilon(p)m}$ - they both equal $(-1)^m$ if $p$ is not 1 modulo 4, and both equal $1$ if $p$ is 1 modulo 4. Therefore, the Hasse invariant being 1 forces the discriminant of $\Lambda_1$ to be 1, thereby proving the theorem.
\end{proof}

\subsection{The 2-adics}
Let $k = \Q_2$. We state the 2-adic analogue of Proposition \ref{pstoll} (again, see \cite{Stoll}):
\begin{Proposition}\label{2stoll}
Let $J$ be the Jacobian of a genus $g$ hyperelliptic curve. Then the quantity $2^g=b_{2}=\displaystyle{\frac{J(\Q_{2})/2J(\Q_{2})}{J[2](\Q_{2})}}$ depends only on $g$, and not on $J$.
\end{Proposition}

We now prove Proposition \ref{solintorb2}. As in the proof of Theorem \ref{psolintorb}, we will use the existence of the ideal $I_1$ as proved in \cite{BG}, and deduce the existence of $I_2$. We need the divisibility condition because Bhargava and Gross need them to deduce the existence of $I_1$ (we would not need these conditions to deduce the existence of $I_2$, if we were guarenteed the existence of $I_1$). 
\begin{proof}[Proof of Proposition \ref{solintorb2}]
Suppose that $\alpha \in (L^{\times}/L^{\times2})_{N=1}$ corresponds to a 1-soluble $\Q_2$ orbit. The first part of the proof proceeds along exactly the same lines - we will still use \cite{BG} for the existence of $I_1$. We still have $I_1$ is self dual for the bilinear form $B_1 = \Bij_{\alpha}$, and therefore $I_1$ and $\gamma^{-1}I_1$ are dual to each other for the form $B_2 = \Bij_{\gamma \alpha}$. We apply the corresponding result of Cassels \cite[Lemma 4.1]{Cassels} to $\Z_2$ to find a suitable basis $(f_i)$ to express the Gram matrix $B_2(f_i,f_j)$ of $B_2$ in a suitable form. Cassels' result is in terms of the quadratic form associated to $B_2$ - translating this in terms of the bilinear form $B_2$ we have  
\begin{equation*}
B_2(f_i,f_j)=
\left(
\begin{array}{cccc}
2^{b_1}Q'_1&&& \\
&2^{b_2}Q'_2 &&\\
&&\ddots &\\
&&&2^{b_n}Q'_{k'}
\end{array}
\right)
\end{equation*} 
where each $b_i \geq 0$ and $Q'_i$ is either a $1 \times 1$ block consisting of $u_i \in \Z_2^{\times}$, or 
$$
Q'_i = H =  
\left(
\begin{array}{cc}
&1\\
1&
\end{array}
\right),$$ 
or
$$Q'_i = H_0 = 
\left(
\begin{array}{cc}
2&1\\
1&2
\end{array}
\right).$$ 
By construction, we know that the bilinear forms $B_1$ and $B_2$ are split. Therefore, by Proposition \ref{classint}, it suffices to find a lattice containing the $\Z_2$-span of the $f_i$ which is self-dual for $B_2$. By multiplying the basis vectors $f_i$ by appropriate negative powers of $2$ we may assume that all the $b_i$ are either $0$ or $1$. In fact, we may assume that $b_i =0$ if $Q' = H$ or $H_0$. Indeed, if a 2-dimensional vector space with basis $e_1,e_2$ has a bilinear form with Gram matrix $2H$ or $2H_0$, by replacing $e_1$ with $e_1/2$ (and leaving $e_2$ unchanged), we are left with a lattice that is self dual. In the first case the Gram matrix with respect to the new basis would be $H$. In the second case, the Gram matrix with respect to the new basis will be 
$$\left(
\begin{array}{cc}
1&1\\
1&4
\end{array}
\right).$$
The Gram matrix now has the form
$$\left(
\begin{array}{ccccccc}
2U_1&&&\\
&Q_2 &&\\
&&\ddots&\\
&&&Q_k
\end{array}
\right)
$$
where $U_1$ is a diagonal matrix of size $2a \times 2a$ consisting solely of units, and the $Q_i$ are either $1 \times 1$ or $2\times2$  matricies with unit determinant. The claim on the parity of the size of $U_1$ holds because the discriminant of $B_2$ in $\Q_2^{\times}$ is $-1$ modulo squares, and hence has even 2-adic valuation. The proposition follows from Lemma \ref{two} below. 
\end{proof}
\begin{Lemma}\label{two}
Let $\Lambda = \Z_2f_1 \oplus \Z_2f_2$ be equipped with a bilinear form whose Gram matrix in the basis $(f_1,f_2)$ is
$$\left(
\begin{array}{cc}
2u_1&0\\
0&2u_2
\end{array}
\right),$$where $u_1$ and $u_2$ are units in $\Z_2$. Then there exists a lattice $\Lambda' \supset \Lambda$ which is self dual for $B$. 
\end{Lemma}
\begin{proof}
The lattice $\Lambda'$ spanned by $(f_1 + f_2)/2$ and $(f_1 - f_2)/2$ has the required properties. 
\end{proof}

\subsection{Archimedean fields}
\subsubsection*{The complex numbers}
In the case $k = \C$, we work over an algebraically closed field, and so for every $c \in \Invrs$, there is precisely one $\Gij(\C)$ orbit with invariants $c$.
\subsubsection*{The real numbers}
We have the archimedean version of Propositions \ref{pstoll} and \ref{2stoll} (again, see \cite{Stoll}):
\begin{Proposition}\label{inftystoll}
Let $J$ be the Jacobian of a genus $g$ hyperelliptic curve. Then the quantity $2^{-g}=b_{\infty}=\displaystyle{\frac{J(\R)/2J(\R)}{J[2](\R)}}$ depends only on $g$, and not on $J$.
\end{Proposition}

Let $c \in \Invrs(\R)$, and let $f_c$ denote the corresponding polynomial. Define $\Invrs(\R)^{(a,b)}$ to be the set of $c \in \Invrs(\R)$ such that $f_c$ has $a$ pairs of complex conjugate roots, $b$ positive real roots, and $n-2a + b$ negative real roots. The stabilizer $\Stab_c$ (as a group scheme over $\R$) depends only on the $(a,b)$ such that $c \in \Invrs(\R)^{(a,b)}$. We call this $\Stab^{(a,b)}$. Let $\tauab$ denote the size of this group. 

Further, computing with the descent map for the Jacobians $J_1$ and $J_2$ shows that the number of 1-soluble orbits with invariants $c$ depends only on $(a,b)$. The same holds for $(1,2)$-soluble orbits. 

\subsection{Orbits over $\Q$ and $\Z$}
For the Jacobian $J$ of a hyperelliptic curve over $\Q$, the local constants $b_{\nu}$ clearly mulitply to yield 1. We conclude this section by demonstrating the existence of $\Z$-representatives of locally soluble $\Q$-orbits:
\begin{proof}[Proof of Theorem \ref{intorb}]
The split groups $\SO_n$ have class number 1. Therefore, $H = \SO(V_1)\times \SO(V_2)$ also has class number one. The same proof as in \cite{BS2selm} applies.
\end{proof}

\section{Counting integral orbits}
In this section, we use Bhargava's averaging technique to count the number of $\Gij(\Z)$-orbits on $\Vij(\Z)$ of height bounded by $X$, as in \cite{BS2selm}. Towards that end, we first define a height function on $\Vij(\R)$. Recall that $\Inv$ is the categorical quotient of $\Vij$ by $\Gij$, i.e. functions on $\Inv$ are $\Gij$ invariant functions on $\Vij$. The scheme $\Inv$ is defined over $\Z$, and equals $\Spec(\Z[a_1, \hdots a_{n-1},e])$. A point $c=(a_1, \hdots, a_{n-1},e)$ corresponds to the polynomial $f(x) = x^n + a_1 x^{n-1} + \hdots a_{n-1}x + e^2$. We define a height function on $\Inv(\R)$ as follows: 
$$\h(c) = \h(f) = \max \{ |a_i|^{1/2i},|e|^{1/n}\}.$$

The canonical map $\pi: \Vij \rightarrow \Inv$ is given by sending the operator $T$ to the non-constant coefficients of its characteristic polynomial, and the pfaffian. The height function $\h:\Vij(\R) \rightarrow \R$ (we use the same symbol, as for the height function on $\Inv(\R)$) is defined to be the composition of $\pi$ and $\h: \Inv(\R) \rightarrow \R$.

Notice that $\h$ is homogenous of degree 1 on $\Vij$, i.e. $\h(\lambda T) = \lambda \h(T)$, for $\lambda \in \R$ positive, $T \in \Vij(\R)$.
\subsection{Fundamental domains}
Let $\Vij(\R)_{\sol}$ denote the subset of $\Vij^{\textrm{rs}}(\R)$ consisting of elements which the property $\sol$, where $\sol$ either stands for 1-soluble or $(1,2)$-soluble. Depending on whether we want to count 1-soluble orbits or (1,2)-soluble orbits, we will choose $\sol$ to be 1-soluble, or (1-2)-soluble respectively. 

We partition $\Vij(\R)_{\sol}$ into sets indexed by $(a,b)$ with $2a + b \leq n$ as in \S 6.4:
$$\Vij(\R)_{\sol} = \bigcup_{(a,b)}\Vij(\R)^{(a,b)}_{\sol}$$
Here  $\Vij(\R)_{\sol}^{(a,b)}$ consists of $T \in \Vij(R)_{\sol}$ such that $\pi(T) \in \Inv(\R)^{(a,b)}$. Recall that the stabilizer of $T \in \Vij(\R)^{(a,b)}_{\sol}$ is independent of $T$. We have defined this group to be $\Stab^{(a,b)}$, and its size to be $\tauab$.

Similarly, let $\Vij(\Z)^{(a,b)} = \Vij(\Z) \cap \Vij(\R)_{\sol}^{(a,b)}$ (we have supressed the subscript $_{\sol}$). If we are working in the case where $\sol$ means 1-soluble, then we will acknowledge this with the notation $\Vij(\Z)_1^{(a,b)}$. Similarly, if $\sol$ stands for $(1,2)$-soluble, we will use the notation $\Vij(\Z)_{(1,2)}^{(a,b)}$. 
\subsubsection*{Fundamental sets for the action of $\Gij(\R)$ on $\Vij(\R)_{\sol}$}
We use a Kostant section (recall that there are two distinguished orbits, and therefore two different Kostant-sections) $\kappa: \Inv(\R) \rightarrow \Vij(\R)$ to define a fundamental set for the action of $\Gij(\R)$ on $\Vij(\R)_{\sol}^{(a,b)}$. Let $\Invrs(\R)^{(a,b)}$ denote the intersection of $\pi(\Vij(\R)^{(a,b)})$ and $\Invrs(\R)$. The number of $\Gij(\R)$-orbits having property $\sol$ depends only on $(a,b)$. Denote this number by $\tau^{(a,b)}_{\sol}$. Just as in \cite{BG}, there exist elements $h_1, \hdots, h_{\tau^{(a,b)}_{\sol}} \in \GL(V_1 \oplus V_2)$ such that $\D_{\sol}'^{(a,b)} = \bigcup_{i}h_i \kappa(\Inv(\R)^{(a,b)})h_i^{-1}$ is a fundamental set for the action of $\Gij(\R)$ on $\Vij(\R)^{(a,b)}$. 
We work with the fundamental set $\D_{\sol}^{(a,b)}$ where
$$\D_{\sol}^{(a,b)} = \R_{>0}\{T \in \D_{\sol}'^{(a,b)}: \h(T) = 1\}.$$
Notice that the size of the entries of any $T \in \D_{\sol}^{(a,b)}$ having height $X$ is bounded by $O(X)$. This is because $\{T \in \D_{\sol}'^{(a,b)}: \h(T) = 1\}$ is a bounded set. Let $\D_{\sol}^{(a,b)}(X)$ denote the subset of $\D_{\sol}^{(a,b)}$ consisting of elements having height bounded by $X$. 

For ease of notation, we will supress the subscript $_{\sol}$ while referring to $\D_{\sol}^{(a,b)}(X)$. For the better part of what follows, the results stay the same whether $\sol$ stands for 1-soluble, or (1,2)-soluble. We will revert to using the subscript only when it matters what $\sol$ stands for. We will then refer to $\Dabs(X)$ as $\Dab_1(X)$ if we are working with 1-soluble orbits, and to $\Dabs(X)$ as $\Dab_{1,2}(X)$ if we are working with (1,2)-soluble orbits.

\subsubsection*{A fundamental domain for the action of $\Gij(\Z)$ on $\Gij(\R)$}
We describe a fundamental domain $\Fu$ for the left action of $\Gij(\Z)$ on $\Gij(\R)$ as constructed by Borel in \cite{Borel}.

The set $\Fu$ may be expresed in the form $\Fu_1 \times \Fu_2$, where $\Fu_i$ is the fundamental domain for the action of $\SO(V_i)(\Z)$ on $\SO(V_i)(\R)$, $i\in \{1,2\}$. $\Fu_1$ may be expressed as 
$$ \Fu_1 = \{ur'\theta \ : \ u \in N_1'(r'),\ r' \in R', \ \theta \in K_1 \}$$ 
where $N'_1(r')$ is an absolutely bounded measurable set - which depends on $r' \in R$ - of unipotent lower triangular matricies; $R'$ is the subset of the torus of diagonal matricies with positive entries 
$$ \left[
\begin{array}{ccccccc}
r'^{-1}&&&&&& \\
&\ddots&&&&&\\
&&r_m'^{-1}\\
&&&1&&&\\
&&&&r'_m &&\\
&&&&&\ddots&\\
&&&&&&r'_1
\end{array}
\right]
$$
constrained by the relations $r'_1/r'_2 > c,\ \hdots, \  r_{m-1}/r'_m >c,\  r'_m>c$; and $K_1$ a maximal compact of $\SO(V_1)(\R)$. We will parameterize $R'$ differently. Indeed, for $i = 1,\ \hdots,\ m$ let $r_i = r'_i\hdots r'_m$. Then $r = (r_1, \hdots,r_m)$ belongs to $R'$ exactly when $r_i > c$. We define $\Fu_2$ similarly, and we denote by $S$ the diagonal torus, and by $s_i$, the analogus coordinates. Let $N'(t)$ denote the product of $N'_1(r) \times N'_2(s)$, and let $K$ denote the maximal compact subgroup of $\Gij(\R)$ given by the products of the $K_i$. We fix a Haar measure $dh$ on $\Gij(\R)$ by setting 
$$dh = \prod_{i = 1}^k(r_is_i)^{i^2 - 2im}\cdot du\cdot d^{\times}t \cdot d\theta,$$ 
where $du$ is an invariant measure on $N$, the group of unipotent lower triangular real matricies, $d\theta$ is the unique haar measure on $K$ giving it unit volume, and $d^{\times}t = d^{\times}rd^{\times}s = \frac{dr}{r} \frac{ds}{s}$ is a Haar measure on $R \times S$.

\subsubsection*{A fundamental domain for the action of $\Gij(\Z)$ on $\Vij(\R)_{\sol}$}
For $h \in \Gij(\R)$, we regard $\Fu h\cdot \Dab(X)$ as a multiset, where the multiplicity of $T \in \Fu h \cdot \Dab(X)$ is given by $\#\{h' \in \Fu : T \in h'h \cdot \Dab(X) \}$. The $\Gij(\Z)$ orbit of $T$ is counted with multiplicty $\displaystyle{\frac{\Stab_T(\R)}{\# (\Stab_T(\R) \cap \Gij(\Z))}}$. Let $\Stab_T(\Z)$ denote the denominator in the previous expression. 

The same argument as in \cite[\S 4.2]{SW} applies to conclude that $(\Stab_T(\R) \cap \Gij(\Z))$ is non-trivial only for a measure-zero set of $\Vij(\R)$. Recall that the group scheme $\Stab_T$ is constant over $T\in \Vij^{(a,b)}$, and $\tauab$ denotes the cardinality of $\Stab_T(\R)$. Therefore, the multiset $\Fu h \cdot \Dab(X)$ is a cover of a fundamental domain for $\Gij(\Z)$ on $\Vij(\R)^{(a,b)}$ of degree $\tauab$.
\subsection{Counting the number of integral orbits}
\begin{Definition}
An element $\alpha \in \Vij(\Q)$ is called irreducible if $\pi(\alpha) \in \Invrs(\Q)$ and if $\alpha$ is not distinguished.
\end{Definition}
For a $\Gij(\Z)$-invariant set $S \subset \Vij(\Z)^{(a,b)}$, define $N(S,X)$ to be the number of irreducible $\Gij(\Z)$ orbits of $S$ that have height bounded by $X$, where each orbit $\Gij(\Z) \cdot T$ is weighted by $1/\#\Stab_T(\Z)$. 
\begin{Theorem}\label{countint}
$$\displaystyle{N(\Vij(\Z)^{(a,b)},X) = \frac{1}{\tauab}\Vol(\Fu \cdot\D^{(a,b)}(X)) + o(X^{{n^2}})}$$
\end{Theorem}

We will spend most of this section proving Theorem \ref{countint}. By our construction of the fundamental domain, we have 
$$N(S,X) = \frac{1}{\tauab}\#\{\Fu h \cdot \D^{(a,b)}(X) \cap S^{\irr}\}$$
for any $h \in \Gij(\R)$. Let $A_0$ be a bounded open $\Kij$ invariant ball in $\Gij(\R)$. Averaging the above equation over $A_0$ we have: 
\begin{equation}\label{equ}
N(S,X) = \frac{1}{\tauab\Vol(A_0)}\int_{h\in A_0}\#\{\Fu h\cdot \D^{(a,b)}(X) \cap S^{\irr}\}dh .
\end{equation}
We use Equation \ref{equ} to define $N(S,X)$ even if $S$ is not $\Gij(\Z)$-invariant. Using an argument entirely analogous to the proof of \cite[Theorem 2.5]{BS2selm} (Bhargava's averaging technique), we obtain
\begin{equation}\label{equation}
N(S,X) = \frac{1}{\tauab\Vol(A_0)}\int_{h\in \Fu}\#\{ hA_0\cdot \D^{(a,b)}(X) \cap S^{\irr}\}dh
\end{equation}

We now state a result of Davenport \cite{Davenport} which we will use extensively in what follows.
\begin{Proposition}[Davenport]\label{Davlem}
Let $A$ be a bounded, semi-algebraic multiset in $\R^n$ having maximum multiplicity $m$, and that is defined by at most $k$ polynomial inequalities each having degree at most $\ell$. Then, the number of lattice points (counted with multiplicity) contained in the region $A$ is 
$$\Vol(A) + O(max\{\Vol(\overline{A}),1\}),$$
where $\Vol(\overline{A})$ denotes the greatest d-dimensional volume of any projection of $A$ onto a coordinate subspace obtained by equating $n-d$ coordinates to zero, where $d$ takes all values from 1 to $n-1$. The implied constant in the second summand depends only on $n,m,k$ and $\ell$. 
\end{Proposition}

Here is a sketch of how we prove Theorem \ref{countint}. We divide $\Fu$ (the region of integration in Equation \eqref{equation}) into two parts: the main body, and the cuspidal region. We will prove that the integral of $\#\{\Fu h\D(X) \cap S^{\red}\}$ over the main body is $o(X^{n^2})$ (Proposition \ref{redorb}), and that the integral of $\#\{\Fu h\D(X) \cap S^{\irr}\}$ over the cuspidal region is $o(X^{n^2})$ (Proposition \ref{cutcusp}). The result will then follow from an application of Proposition \ref{Davlem}. 
\subsubsection{The number of irreducible points in the cusp is negligible}
Recall that we have fixed bases for $V_1$ and $V_2$ with respect to which the bilinear forms $B_i$ have gram matricies $\pm B$ (Equation \eqref{stdform} in \S2). 

We pick a set of coordinates on $\Vij$ as follows. An element $T$ corresponds to 
$$T= \left[
\begin{array}{c|c}
0&A \\ \hline
A^{*}&0
\end{array}
\right]
$$
where $A$ is some $n \times n$ matrix, and $A^{*} = (-BAB^{-1})^{\textrm{t}}$ (where the superscript refers to taking the transpose). The matrix $A^{*}$ is the unique matrix which makes $T$ self-adjoint. To pick coordinates on $\Vij$, it suffices to pick coordinates on the set of all matricies $A$ which we do as follows:
$$\left[
\begin{array}{cccc}
a_{m \hspace{2pt} \text{-}m}&a_{m \hspace{2pt} 1\text{-}m}&\hdots&a_{\hspace{1pt}mm}\\
a_{m\text{-}1 \hspace{2pt} \text{-}m}&a_{m\text{-}1 \hspace{2.5pt} 1\text{-}m}&\hdots&a_{m\text{-}1 \hspace{1.5pt}m}\\
\vdots&&\ddots&\vdots\\
a_{\text{-}m \hspace{2pt}\text{-}m}&a_{\text{-}m \hspace{2pt} 1\text{-}m}&\hdots&a_{\text{-}m \hspace{1.5pt} m}
\end{array}
\right].
$$
The $a_{ij}$ are scaled by the action of the torus $R \times S$. Define $w_{ij}$ to be the weight according to which $R \times S$ scales $a_{ij}$. For instance, if $i$ and $j$ are positive, $w_{ij} = r_1^{\text{-}1}\hdots r_i^{\text{-}1}s_1^{\text{-}1}\hdots s_j^{\text{-}1}$, and $w_{\text{-}i \hspace{2pt} \text{-}j} = w_{ij}^{-1}$. We define a partial order $\preceq$ on the set of variables: $a_{ij} \preceq a_{i'j'}$ if $i'\le i$ and $j' \le j$. Further, $a_{ij} \preceq a_{i'j'}$  precisely when $\frac{w_{ij}}{w_{i'j'}}$ consists of non-positive powers of $r$ and $s$. With respect to this order, $a_{mm}$ is the unique minimal element. 

We now prove some results on reducibility of elements in $\Vij(\Z)$, which we will need to prove Proposition \ref{cutcusp}. 
\begin{Lemma}\label{red1}
 If for some $i$ the top-right $i \times (2m+2 -i)$ block of $A$ is identically zero, then the corresponding $T$ has discriminant zero.
\end{Lemma}
\begin{proof}
If such a block is identically zero, then the determinant of $A$ is zero. Therefore, the operator $T$ has the eigenvalue $0$ with multiplicity two, rendering it impossible for $T$ to be regular semisimple. 
\end{proof}

\begin{Lemma}\label{red2}
If the top-right $m \times (m+1)$ or $(m+1) \times m$ block of $A$ is zero, then $T$ is reducible.
\end{Lemma}
\begin{proof}
We will show that in the first case, the orbit will be $1$-distinguished, and in the secod case, the orbit will be $2$-distinguished. Indeed, it is easy to check that 
$$T^2= \left[
\begin{array}{c|c}
AA^*&0 \\ \hline
0&A^*A
\end{array}
\right]
$$

 If the top-right $m \times (m+1)$ block of $A$ is identically zero, then an easy computation shows that the top-right $m \times m$ block of $AA^*$ will also be zero. Consequently, the isotropic space $X$ spanned by the last $m$ basis vectors of $V_1$ has the property that $T^2 X \subset X^{\perp}$, and so the orbit is $1$-distinguished. The same proof (except that we use $A^*A$ instead) shows that in the second case, the orbit would be $2$-distinguished. Therefore, $T$ is reducible. 
\end{proof}

\begin{Lemma}\label{red3}
Suppose that $i + j = 2m+1$. If the top-right $i \times j$  and $j \times i$ blocks of $A$ are zero, then $T$ has discriminant zero.
\end{Lemma}
\begin{proof}
Without any loss of generality, we assume that $i < j$. If $A$ has the property that the top-right blocks of size $i \times j$ and $j \times i$ are identically zero, then so does $A^*$, and therefore so do $AA^*$ and $A^*A$. We will show that $AA^*$ (and similarly, $A^*A$) has repeated eigenvalues. For brevity, let $Y$ denote the matrix $AA^*$. It will be of the form 
$$Y= \left[
\begin{array}{c|c|c}
Y_{11}&\textbf{0}&\textbf{0} \\ \hline
*&Y_{22}&\textbf{0} \\ \hline
*&*&Y_{33}
\end{array}
\right]
$$
where $Y_{11}$ and $Y_{33}$ are of size $i \times i$ and $Y_{22}$ is of size $(j-i)\times (j-i)$. The determinant of $Y$ equals the product of the determinants of $Y_{11}$, $Y_{22}$ and $Y_{33}$. Note that the determinant of $Y_{11}$ equals the determinant of $Y_{33}$. This is so because $Y$ is self-adjoint for the bilinear form on $V_1$. The same conclusions hold with the characteristic polynomial replacing the determinant, because the same argument applies to $Y - xI$. Let $P_M$ denote the characteristic polynomial of any square matrix $M$. We have shown that $P_Y = P_{Y_{22}}P_{Y_{11}}^2$, and so $AA^* = Y$ must have repeated eigenvalues, as claimed. Therefore, $T$ has discriminant zero and is reducible. 
\end{proof}

Define the cusp, or cuspidal region to be the set of all elements of $\Vij(\R)$ such that $|a_{mm}| <1$, and define the main body to be the complement of the cuspidal region. We say that $h \in \Fu$ is cuspidal if $hA_0\D^{(a,b)}(X)$ lies fully in in the cusp. Clearly, an integral element will lie in the cusp only if $a_{mm} =0$. We have:
\begin{Proposition}\label{cutcusp}
$$\int_{h\in \Fu}\#\{ hA_0\D^{(a,b)}(X) \cap \Vij(\Z)_{a_{mm}=0}^{\irr}\}dh = o(X^{n^2})$$
\end{Proposition}
\begin{proof}
To lighten notation, we drop the superscript $^{(a,b)}$ while proving this result. The strategy is to use Proposition \ref{Davlem} to replace the number of integral points with a volume. The same argument as in \cite[Proposition 4.5]{SW}, shows that it suffices to prove 
$$\int_{t \in R \times S}\#\{ tA_0\D(X) \cap \Vij(\Z)_{a_{mm}=0}^{\irr}\}\delta(t) \dxt = o(X^{n^2 }).$$
If $t = (r,s) \in R \times S$ has the property that $tA_0\D(X) \subset \Vij(\R)_{|a_{ij}|<1}$, then $tA_0\D(X) \subset \Vij(\R)_{|a_{i'j'}|<0}$ for all $a_{i'j'} \preceq a_{ij}$. This is true because $\frac{w_{i'j'}}{w_{ij}}$ consists of non-positive powers of $r$ and $s$, and so $t$ would act with a higher negative weight on $a_{i'j'}$. 

Let $U$ denote a subset of the coordinates $a_{ij}$, with the property that if $a_{i_0j_0} \in U$ then $U$ also contains $a_{ij}$ with $a_{ij} \preceq a_{i_0j_0}$, i.e. $U$ contains all variables $a_{ij}$ to the top-right of $a_{i_0 j_0}$. Define $\Vij(U)$ to be the subspace of $\Vij$ given by $a_{ij} = 0$, $a_{ij}\in U$. Define $\Vij(U)(\Z)^{\irr}_0$ to be $\alpha \in \Vij(U)(\Z)^{\irr}$ such that $a_{ij} \neq 0$ for $a_{ij} \notin U$. It suffices to prove that
$$\int_{t\in R \times S}\#\{ tA_0\D(X) \cap \Vij(U)(\Z)^{\irr}_0\}\delta(t)\dxt = o(X^{n^2})$$
By Lemmas \ref{red1} and \ref{red2}, if $U$ contains $a_{ij}$ with $i+j \le 0$ or with $\{i,j\} = \{0,1\}$, then every element of $\Vij(U)(\Z)$ is reducible. Similarly, by Lemma \ref{red3}, if $U$ contains $a_{ij}$ and $a_{ji}$ for some pair $(i,j)$ such that $i + j =1$, then every element of $\Vij(U)(\Z)$ is reducible. We may assume that $U$ doesn't contain any such $a_{ij}$. We now prove two results which we will need to finish this proof.
\begin{Claim}\label{claim2}
For any $w_{ij} \notin U$, we may assume that $Xw_{ij}(t) \geq 1$. 
\end{Claim}
\begin{proof}
If $Xw_{ij}(t) <1$, then $a_{ij}(\alpha) = 0$ for $\alpha \in \Vij(U)(\Z)$, and so $\alpha \notin \Vij(U)(\Z)_0$. 
\end{proof}
It is an easy consequence of Claim \ref{claim2} that $Xs_m^{-1} > 1$ and $Xr_m^{-1}>1$, and for all $k$, $X^2s^{-1}_k >1$ and $X^2r^{-1}_k >1$.

\begin{Claim}\label{claim1}
We have 
$$\#\{ tA_0\D(X) \cap \Vij(U)(\Z)^{\irr}_0\} = O\Big(\displaystyle{\prod_{a_{ij} \in U}w_{ij}^{-1}(t)X^{n^2 -|U|}}\Big).$$
\end{Claim} 
\begin{proof}
By Proposition \ref{Davlem}, 
$$\#\{ tA_0\D(X) \cap \Vij(U)(\Z)^{\irr}_0\} = O(\Vol) + E, $$
where $\Vol$ is the volume of the projection of $tA_0\D(X)$ onto $\Vij(U)$, and $E$ is error term mentioned in Proposition \ref{Davlem}. If the projection of $tA_0\D(X)$ onto some line spanned by $a_{ij}$, for some $a_{ij} \notin U$ has volume less than $1$, then $ tA_0\D(X) \cap \Vij(U)(\Z)^{\irr}_0 = \emptyset$. Therefore, the volume of the projection of $tA_0\D(X)$ onto any coordinate subspace of $\Vij(U)$ is bounded by $\Vol$. We therefore have the bound
$$\#\{ tA_0\D(X) \cap \Vij(U)(\Z)^{\irr}_0\} = O(\Vol).$$
The claim follows from the fact that $\Vol =  O\Big(\displaystyle{\prod_{a_{ij} \in U}w_{ij}^{-1}(t)X^{n^2 -|U|}}\Big)$.
\end{proof}

By Claim \ref{claim1}, we have
$$
\begin{array}{rcl}
&&\displaystyle{\int_{t\in R \times S}\#\{ tA_0\D(X) \cap \Vij(U)(\Z)_0^{\irr}\}\delta(t)\dxt} \\
&=& \displaystyle{O\bigg(X^{n^2-|U|}\int_{t\in R \times S} \prod_{a_{ij} \in U}w_{ij}^{-1}\prod (r_is_i)^{i^2-2im} \dxt\bigg)}.
\end{array}
$$
It remains to prove that $\displaystyle{\int_{t\in R \times S} \prod_{a_{ij} \in U}w_{ij}^{-1}\prod (r_is_i)^{i^2-2im} \dxt} = o\big(X^{|U|}\big)$. Suppose first, that $U$ is contained in the top-right $m \times m$ block. We then have  $\prod_{i,j>0}w_{ij}^{-1} = \prod(r_is_i)^{im}$. If $U$ is strictly contained in the top-right $m \times m$ block, all the exponents in the $r_i$ and $s_i$ (in the integrand) are strictly negative, and so the integral is bounded by $O(1)$. If $U$ equals the top-right $m\times m$ block, the exponents of $r_m$ and $s_m$ are zero, and the other exponents are strictly negative. By Claim \ref{claim2}, $Xw_{10} \geq 1$ and $Xw_{01} \geq 1$. We make the exponents of $r_m$ and $s_m$ negative as well, by multiplying the entire expression by $X^2w_{10}w_{01}$, thus bounding the integral by $o(X^2)$. 

Let us therefore assume that $U$ is not contained in the top-right $m\times m$ block. It follows that $U$ either contains a variable of the form $a_{mj}$ with $j \leq 0$ of $a_{im}$, $i \leq 0$. We will induct on $m$ to prove the proposition. We deal with the problem in two cases. Case 1 will be when $U$ contains variables of both forms, i.e. $a_{mj}$ and $a_{im}$ with $i,j \leq 0$, and Case 2 is when $U$ contains variables of only one of the two kinds. The proof of the first case is strictly harder than the second case, so we will be content with simply proving the first case. 

\paragraph{Case 1}
Let the topmost row $U_R$ of $U$ have size $1 \times (m+b_s)$, and the rightmost column $U_C$ have size $(m+b_r) \times 1$, with $b_s \geq b_r$. As mentioned above, using Lemma \ref{red3} we assume that, $b_r < m$. In order to apply induction, we have to show that 
$$\delta_m(r,s) \delta_{m-1}(r,s)^{-1}\displaystyle{\prod_{w \in U_R \cup U_C} w^{-1}(r,s)} = o(X^{2m + b_s + b_r -1})$$
For ease of notation, let $a = b_s - b_r$, $k_s = m-b_s$ and $k_r = m-b_r$. The left hand side of the above equation is going to be a product of a term $s_u$ consisting of the parameters $s_k$ and a term $r_u$ consisting of the parameters $r_k$. 

A calculation shows that 
$$\displaystyle{s_u = s_m^{1-a}\hdots s_{k_s+1}^{1-a}s_{k_s}^{-a} \hdots s_1^{k_s -1 +1}}$$
and 
$$r_u =r_m^{1+a} \hdots r_{k_s + a + 1}^{1+a}r_{k_s +a}^a \hdots r_{k_s+1} r_{k_s}^0 \hdots r_1^{-k_s + 1}.$$
 Recall that $Xw >1$ for any $w \notin U$. Let $w_0 = (r_1 \hdots r_m)^{-1}(s_{k_s+1} \hdots s_m)$, be the weight of the variable $a_{m \hspace{.03in}\text{-}b_s}$. We have that $w_0^a s_u r_u$ equals $s_ms_{m-1}\hdots s_{k_s+1}r_m \hdots r_{k_s+a+1}$ multiplied by non-positive powers of the $s_k$ and $r_k$. Recall that we have $X^2s_k^{-1} >1$ and $X^2r_k >1$. Further, $Xr_m^{-1} >1$ and $Xs_m^{-1}>1$. 

Multiplying $s_ur_u$ by $X^{a + 2b_s + 2(b_s -a) -2}w_0^a (s_m \hdots s_{k_s +2})^{-1}(r_m \hdots r_{k_s + a +1})^{-1}$ gives us a product of non-positive powers of the $r_k$ and $s_k$. Further, it is easy to see that the exponent of $X$ is at most $2m + b_s + b_r -2$. Finally, we again use that $X^4r_k^{-1}s_k^{-1} <1 $ to multiply by a small power of $X$ to make the exponents of $r_k$ and $s_k$ negative. This concludes the first case. 

As we mentioned above, the proof of the second case is simpler and runs along the same lines. We have thus proved Proposition \ref{cutcusp}.
\end{proof}

\subsubsection{Proof of Theorem \ref{countint}}
We largely follow the exposition in \cite[Theorem 4.9]{SW}. Let $\Fu' \subset \Fu$ be the set of cuspidal elements of $\Fu$. By Proposition \ref{cutcusp}, we have 
$$
\begin{array}{rcl}
&&N(\Vij(\Z)^{(a,b)},X)\\
&=&\displaystyle{\frac{1}{\tauab\Vol(A_0)}\int_{h\in \Fu}\#\{ hA_0\D^{(a,b)}(X) \cap \Vij(\Z)^{\irr}\}dh}\\
&=& \displaystyle{\frac{1}{\tauab\Vol(A_0)}\int_{h\in \Fu \setminus \Fu'}\#\{ hA_0\D^{(a,b)}(X) \cap \Vij(\Z)^{\irr}\}dh}\hspace{3pt} + o(X^{n^2})
\end{array}
$$
By Proposition \ref{redorb}, we may replace $\Vij^{\irr}$ with $\Vij$. We use Proposition \ref{Davlem} to approximate $\#\{ hA_0\D^{(a,b)}(X) \cap \Vij(\Z)\}$.

By construction of $\mathcal{F}'$, the length of the projection of $hA_o\mathcal{D}(X)$ onto the coordinate $a_{mm}$ for all $h$ in $\mathcal{F} \setminus \mathcal{F}'$ is at least $1$.  Further, the weight of $a_{mm}$ being minimal, the volume of all smaller dimensional projections of $hA_o\mathcal{D}(X)$ are bounded by the volume of the projection onto the hyperplane $a_{mm} = 0$. Therefore, $N(\Vij(\Z)^{(a,b)},X)$ equals
$$ \displaystyle{\frac{1}{\tauab\Vol(A_0)}\int_{h\in \Fu \setminus \Fu'}\bigg[\Vol(hA_0\D^{(a,b)}(X)) + O\bigg(\frac{\Vol(hA_0\D^{(a,b)}(X))}{Xw_{mm}}\bigg)\bigg] dh}  + o(X^{n^2})$$
We have that $\displaystyle{\int_{h \in \Fu \setminus \Fu'}\frac{1}{w_{mm}}dh}$ is bounded by $O(1)$. Further, by the same argument used in \cite{SW}, the volume of the cuspidal region $\mathcal{F}'$ is also bounded by $o(1)$. Therefore, we have  
$$
\begin{array}{ccc}
N(\Vij(\Z)^{(a,b)},X) &= &\displaystyle{\frac{1}{\tauab\Vol(A_0)}\int_{h\in \Fu}\Vol(hA_0\D^{(a,b)}(X))dh} +o(X^{n^2})\\
& = &\displaystyle{\frac{1}{\tauab\Vol(A_0)}\int_{h\in A_o}\Vol(\mathcal{F}h\D^{(a,b)}(X))dh} +o(X^{n^2})
\end{array}
$$
The set $\mathcal{F}h\D^{(a,b)}(X)$ does not depend on $h$, and so the integrand equals $\Vol(\mathcal{F}\D^{(a,b)}(X))$. Substituting this in the final equality, we see that 
$$\displaystyle{N(\Vij(\Z)^{(a,b)},X) = \frac{1}{\tauab}\Vol(\Fu \D^{(a,b)}(X)) + o(X^{n^2})}$$ 
as required. 

\subsection{Congruence conditions}
Let $\Lat \subset \Vij(\Z)$ be a subset defined by congruence conditions modulo finitely many prime powers. We want to count irreducible $\Gij(\Z)$-orbits in $\Lat$, and the main result in this subsection is: 
\begin{Theorem}\label{congcount}
We have 
$$\displaystyle{N(\Lat^{(a,b)},X) =  N(\Vij(\Z)^{(a.b)},X) \prod_p \mu_p(\mathcal{L})+ o(X^{n^2})}$$
where $\Lat^{(a,b)} = \Lat \cap \Vij(\Z)^{(a,b)}$, and $\mu_p$ is the $p$-adic density of $\mathcal{L}$ in $\Vij(\Z)$.
\end{Theorem}
The structure of our proof shall be thus: we will first prove Lemma \ref{precongcount}, and using it, we will prove Proposition\ref{redorb}.  We will see that Theorem \ref{congcount} follows immediately. We remark that the proofs of Lemma \ref{precongcount} and Proposition \ref{redorb} are independent of Theorem \ref{countint}.
\begin{Lemma}\label{precongcount}
Notation as above. We then have
$$\int_{h\in \Fu \setminus \Fu'}\#\{ hA_0\D^{(a,b)}(X) \cap \Lat \}dh = \prod_p \mu_p(\mathcal{L})\int_{h\in \Fu \setminus \Fu'}\#\{ hA_0\D^{(a,b)}(X) \cap \Vij(\Z) \}dh + o(X^{n^2})$$
\end{Lemma}
\begin{proof}
We follow the proof of \cite[Theorem 35]{BG}. Suppose that $\mathcal{L}$ is defined by congruences modulo some integer $m$. Then $\mathcal{L}$ may be viewed as a disjoint union of translates $\Lat_1 \hdots \Lat_k$ of the lattice $m\Vij(\Z)$. To estimate $\#\{ hA_0\D^{(a,b)}(X) \cap \Lat \}$, we again use Propsition \ref{Davlem} and see that $\#\{ hA_0\D^{(a,b)}(X) \cap \Lat_i \} = 1/m^{n^2} \Vol(hA_0\D^{(a,b)}(X))$, up to an error of $o(X^{n^2})$. Summing over $i$, we obtain
$$\int_{h\in \Fu \setminus \Fu'}\#\{ hA_0\D^{(a,b)}(X) \cap \Lat \}dh = k/m^{n^2}\int_{h\in \Fu \setminus \Fu'}\Vol(hA_0\D^{(a,b)}(X))dh + o(X^{n^2}).$$
The lemma follows from the observation that the product of the $p$-adic densities of $\Lat \subset \Vij(\Z)$ equals $k/m^{n^2}$. 
\end{proof}

The proof of Theorem \ref{congcount} runs along the same lines as the proof of Theorem \ref{countint}. The only additional input is Lemma \ref{precongcount}, which has to be applied at the obvious point. 

\subsubsection*{Computations modulo $p$ and consequences}

\begin{Lemma}\label{redmod}
The ratio of the number of reducible elements in $\Vij(\F_p)$  to the total number of elements in $\Vij(\F_p)$ is bounded away from $1$ independent of $p$.
\end{Lemma}
\begin{proof}
It suffices to show that the ratio of irreducible elements in $\Vij(\F_p)$ to the total number of elements is bounded away from $0$. 

The total number of invariants modulo $p$ is $p^n$. The cardinality of $\Gij(\F_p)$ is at least $p^{n^2 -n}/2$ for large enough $p$. The stabilizer of any regular semisimple element has at most $2^n$ elements. The cardinality of $\Vij(\F_p)$ is $p^{n^2}$. We define a set of invariants $a_i,e$ to be ``good'' if there exists an orbit with these invariants which is regular semisimple. It suffices to prove that the ratio of the number of good invariants to the total number of invariants is bounded away from zero. Indeed, if there were $N$ good invariants, then there would be at least $Np^{n^2 - n}/2^{n + 1}$ irreducible elements in $\Vij(\F_p)$. If $\frac{N}{p^n} > r_n$, then $\frac{Np^{n^2 - n}/2^{n + 1}}{p^{n^2}} > r_n/2^{n+1}$.

The proportion of polynomials $f(x)$ of degree $n$ which have at least three irreducible factors, and which also have non-zero discrimant, and whose constant term a non-zero square, is positive and bounded away from $0$ independent of $p$. Let $r_n$ be some positive lower bound for the above proportion for all $p$. For invariants giving such polynomials, the number of $\F_p$-orbits is at least $4$. This is because $L = \F_p[x]/(f(x))$ will be a product of at least three fields, and so $|( L^{\times}/L^{\times2})_{N=1}| \geq 4$. Such invariants have to be good, because there have to be at least $2$ irreducible $\F_p$-orbits. The lemma follows. 
\end{proof}

\begin{Proposition}\label{redorb}
We have  
$$\int_{h\in \Fu \setminus \Fu'}\#\{ hA_0\D(X) \cap \Vij(\Z)^{\irr} \}dh = \int_{h\in \Fu \setminus \Fu'}\#\{ hA_0\D(X) \cap \Vij(\Z) \}dh + o(X^{n^2})$$
\end{Proposition}
\begin{proof}
It suffices to prove
$$\int_{h\in \Fu \setminus \Fu'}\#\{ hA_0\D(X) \cap \Vij(\Z)^{\red} \}dh = o\bigg(\int_{h\in \Fu \setminus \Fu'}\#\{ hA_0\D(X) \cap \Vij(\Z) \}dh\bigg).$$ 
To that end, fix $Y \in \N$, some positive number, and let $\Lat_Y \subset \Vij(\Z)$ be the set of all elements whose reduction modulo $p$ is reducible in $\Vij(\F_p)$, for $p \le Y$. For every $Y \in \N$, the set $\Lat_Y$ is defined by congruence conditions and contains $\Vij(\Z)^{\red}$. 

 By Lemma \ref{precongcount}, we have  
$$\int_{h\in \Fu \setminus \Fu'}\#\{ hA_0\D(X) \cap \Lat_Y \}dh = \prod_{p<Y}\mu_p(\Lat_Y)\int_{h\in \Fu \setminus \Fu'}\#\{ hA_0\D(X) \cap \Vij(\Z) \}dh.$$
By Lemma \ref{redmod}, each $\mu_p(\Lat_Y)$ is bounded away from $1$ independent of $p$. Therefore, we have $\displaystyle{\lim_{Y \to \infty}\prod_{p<Y}\mu_p(\Lat_Y)} = 0$. 

The proposition follows, because $\Vij(\Z)^{\red} \subset \Lat_Y$ for all $Y$. 
\end{proof}

\begin{Lemma}\label{bigstab}
The proportion of elements in $\Vij(\F_p)$ having non-trivial stabilizer in $\Gij(\F_p)$ is bounded away from $1$ independent of $p$.
\end{Lemma}
\begin{proof}
Using the same argument as in Lemma \ref{redmod}, it suffices to show that the proportion of invariants having the required property is bounded away from 1 independent of $p$. To that end, we remark that the proportion of degree $n$ polynomials which are irreducible with constant term a non-zero perfect square, is bounded away from zero independent of $p$. The proposition follows.
\end{proof}

\begin{Proposition}\label{nostab}
Let $S\subset \Vij(\Z)$ consist of those $T$ with the property that $\Stab_T(\Q)$ is non-trivial. Then $N(S,X) = o(X^{n^2})$.
\end{Proposition}
\begin{proof}
Let $f \in \Z[x]$ be the polynomial associated to $T$. By Corollary \ref{bigstabred}, the stabilizer in $\Gij(\Q)$ is non-trivial if and only if $f$ is not irreducible. Clearly, $f$ is irreducible in $\Q[x]$ only if its reduction modulo $p$ is irreducible for every prime $p$, and by applying Corollary \ref{bigstabred} again, we see that this happens precisely when the stabilizer of $T$ modulo $p$ is trivial. However, by Lemma \ref{bigstab}, the proportion of elements in $\Vij(\F_p)$ having non-trivial stabilizer is bounded away from $1$ independent of $p$. The product of this ratio over all primes diverges to 0. Therefore, the same argument as in the proof of Proposition \ref{redorb} applies to prove our result. 
\end{proof}

\begin{Lemma}\label{Twodist}
The proportion of invariants over $\F_p$ which are either not regular semisimple, or which satisfy the condition that the distinguished orbits are in the same $\Gij(\F_p)$-orbit, is bounded away from $1$ independent of $p$. 
\end{Lemma}
\begin{proof}
It suffices to show that the proportion of invariants over $\F_p$ which are regular semisimple, and such that the distinguished orbits are in different $\Gij(\F_p)$-orbits, is bounded away from zero independent of $p$. The proportion of invariants such that the corresponding polynomial $f$ has $n$ distinct non-zero roots over $\F_p$ is bounded away from zero independent of $p$. The proportion of such $f$ with the properties that at least one root is a perfect square in $\F_p^{\times}$ and at least one root is \textit{not} a perfect square,  is again bounded away from zero, independent of $p$. Therefore, the proportion of invariants with the property that $\gamma$ is not a perfect square in $L^{\times}$ is bounded away from zero, independent of $p$ (here, $L = \F_p[x]/(f(x))$ is the \'etale algebra associated to regular semisimple invariants). If $(-1)$ is a square in $\F_p^{\times}$, then $(-\gamma)$ is clearly not a perfect square either. If $(-1)$ is not a perfect square, then because one of the components of $\gamma$ is a perfect square, that component of $(-\gamma)$ would now cease to be a square. In either case, $(-\gamma)$ is \textit{not} a perfect square. 

The result follows from Corollary \ref{gammasquare}, which states that the two distinguished orbits lie in the same $\Gij(\F_p)$-orbit if and only if $-\gamma$ is a perfect square in $L\x$.
\end{proof}

\begin{Proposition}\label{difdist}
Let $N_X$ denote the number of invariants $c \in \Invrs(\Z)$ with height bounded by $X$ such that the two distinguished orbits in $\Vij(\Q)$ with invariants $c$ lie in the same $\Gij(\Q)$ orbit. Then $N_X = o(X^{n^2})$, i.e. the proportion of $N_X$ to the number of invariants with height bounded by $X$ goes to zero.
\end{Proposition}
\begin{proof}
Fix $c \in \Invrs(\Z) \subset \Invrs(\Q)$, and let $L$ denote the corresponding etale algebra of dimension $n$. By Corollary \ref{gammasquare}, the two distinguished orbits lie in the same $\Gij(\Q)$ orbit, precisely when $-\gamma$ is a perfect square in $\L^{\times}$. As $c \in \Invrs(\Z)$, all the data can be reduced modulo all primes $p$. 

If $(-\gamma)$ is a perfect square, then either $c$ modulo $p$ is no longer regular semisimple, or the corresponding $(-\gamma) \in L_p^{\times}$ stays a square (here, $L_p$ is the etale $\F_p$ algebra corresponding to the reduction of $c$ modulo $p$). 

Therefore, if the two distinguished  orbits lie in the same $\Gij(\Q)$ orbit, then the reduction of $c$ modulo $p$ is either not regular semisimple, or the two $\F_p$ distinguished orbits lie in the same $\Gij(\F_p)$ orbit. The set whose cardinality is $N_X$ is cut out by infinitely congruence conditions, whose local densities (by Lemma \ref{Twodist}) are bounded away from $1$. The proposition follows. 
\end{proof}
We conclude this paragraph with a lemma which we will need to prove Theorem \ref{12selm}.
\begin{Lemma}\label{smallonetwo}
Consider the proportion $r_p$ of $c = (a_1, a_2, \hdots, a_{n-1},e)\in \Inv(\F_p)$ with the following properties:
\begin{enumerate}
\item The associated polynomial $f$ satisfies $f(0) =0$, i.e. $e=0$. 
\item $f$ has distinct roots $u_1,\ u_2,\ \hdots u_{n-1},\ u_n=0$ over $\F_p$
\item The product of the non-zero roots is a perfect square.
\end{enumerate}
Then there exists some $r>0$ independent of $p$, such that $r_p > r/p$. 
\end{Lemma}
\begin{proof}
The number of invariants such that the corresponding $f$ has distinct roots over $\F_p$, one of which is zero, is $\displaystyle{\prod_{i=1}^{n-1}(p-i)}/(n-1)!$. Adding the condition that the product of the non-zero roots is a square is the same as scaling this by a factor of $1/2$. The ratio $p\displaystyle{\prod_{i=1}^{n-1}(p-i)}/(n-1)!p^n$ is clearly bounded away from zero independent of $p$. The lemma follows.
\end{proof}
\subsubsection*{Infinitely many congruence conditions}
In this section, we count elements having bounded height in subsets of
$\Inv(\Z)$ and $\Vij(\Z)$ that are defined by certain infinite sets of
congruence conditions.

\begin{Definition}
  Let $\Sigma\subset \Inv(\Z))$ be a set defined by (possibly
  infinitely many) congruence conditions. For a prime $p$, let
  $\Sigma_p$ denote the closure of $\Sigma$ in $\Inv(\Z_p)$. We say
  that $\Sigma$ is {\bf large at p} if $\Sigma_p$ contains every $c$
  such that the reduction of $f_c$ modulo $p$ has no triple root.  The
  set $\Sigma \subset \Inv(\Z)$ is then said to be {\bf large} if it
  is large at all but finitely many primes.
\end{Definition}

We now have the following theorem counting the number of elements
having bounded height in large sets.
\begin{Theorem}\label{last}
Let $\Sigma \subset \Inv(\Z)$ be large. Then
$$\#\{c \in \Sigma \cap \Inv(\R)^{(a,b)}|\h(c) < X\} =
\displaystyle{\prod_p \Vol(\Sigma_p) \Vol(\Inv(\R)_{\h<X}^{(a,b)})}$$
\end{Theorem}
\begin{proof}
Let $Z\subset\Inv$ denote the variety of elements $c$ such that $f_c$
has a triple root. Then $Z$ has codimension $2$.  Let $Y_p\subset
\Inv(\Z)$ denote the set of elements $c$ whose reduction modulo $p$
lie in $Z(\F_p)$. An arguement identical to the proof of \cite[Theorem
  3.5]{Bgeosieve} yields the estimate
\begin{equation*}
\#\bigcup_{p>M}\{c\in Y_p:\h(c)<X\}=O\Bigl(\frac{X^{n^2}\log X}{M}\Bigr).
\end{equation*}
The theorem now follows from the above ``tail estimate'' using
standard sieving arguments (see, for example, the proof of
\cite[Theorem 2.21]{BS2selm}).
\end{proof}

Next, we need a weighted version of Theorem \ref{congcount} that
allows for infinitely many congruence conditions. Let $\phi: \Vij(\Z)
\rightarrow [0,1]$ be a $\Gij(\Z)$ invariant function.  Then let
$N_{\phi}(\Vij(\Z),X)$ denote the number of irreducible
$\Gij(\Z)$-orbits of $\Vij(\Z)$ having height bounded by $X$, where
each orbit $\Gij(\Z).T$ is weighted by
$\displaystyle{\frac{\phi(T)}{\#\Stab_T(\Z)}}$.

\begin{Definition}
A function $\phi: \Vij(\Z) \rightarrow [0,1]$ is said to be defined by
congruence conditions if there exist local functions $\phi_p:
\Vij(\Z_p) \rightarrow [0,1]$ satisfying the following conditions:
\begin{enumerate}
\item For all $T \in \Vij(\Z)$, $\displaystyle{\prod_p(\phi_p(T))}$ converges to $\phi(T)$. 
\item For each prime $p$, $\phi_p$ is locally constant outside a
  closed set of measure zero.
\end{enumerate}
\end{Definition}
\begin{Theorem}\label{infcong}
Suppose $\phi: \Vij(\Z) \rightarrow [0,1]$  is defined by congruence conditions, and that the
$\phi_p$ are $\Gij(\Z_p)$ invariant. Then
$$N_{\phi}(\Vij(\Z),X) \leq N(\Vij(\Z),X)\displaystyle{\prod_p}
\displaystyle{\int_{T\in\Vij(\Z)}}\phi_p(T)dT + o(X^{n^2})$$
\end{Theorem}
If the function $\phi$ were nice enough, we expect that the upper
bound is an equality. We will not make precise in this paper what nice
means (see \cite[\S 4]{SW} for the precise definition), and will
content ourselves by saying that a function which were to pick out
locally $1$-soluble orbits would be nice.
\begin{proof}
For $N \in \mathbb{N}$, define $\phi_{p,N}: \Vij(\Z_p) \rightarrow
[0,1]$ to be $\displaystyle{\phi_{p,N}(T) = \max_{T' \in Z}
  \phi_p(T')}$, where $Z = \{T' : T'\equiv T \mod{p^N}\}$. Because
$\phi_p$ is $\Gij(\Z_p)$ invariant, so is $\phi_{p,N}$. Further, as $\phi_p$
is locally constant (outside a set of measure zero), $\phi_{p,N}$
converges to $\phi_p$ as $N$ goes to infinity. For $Y \in \N$, define
$\psi_Y$ to be $\displaystyle{\prod_{p<Y}\phi_p}$, and
$\displaystyle{\psi_{Y,N}}$ to be
$\displaystyle{\prod_{p<Y}\phi_{p,N}}$.

The same method used to prove Theorem \ref{congcount} applies to prove
this result (with equality instead of an upper bound) for the function
$\psi_{Y,N}$. Therefore, we have
$$N_{\phi}(\Vij(\Z),X) \leq N_{\psi_{Y,N}}(\Vij(\Z),X) =
N(\Vij(\Z),X)\prod_{p<Y} \int_{T\in\Vij(\Z)}\psi_{p,N}(T)dT +
o(X^{n^2}).$$ The theorem follows by allowing $N$, and then $Y$, to go
to infinity.
\end{proof}

\section{Selmer Groups}
In this section, we prove a strengthening of Theorems \ref{2selm} and \ref{12selm}. Recall that for $c = (a_1, \hdots a_{n-1}, e) \in \Inv$, we have associated a polynomial $f_c(x)$ and hyperelliptic curves whose affine equations are $y^2 = f_c(x)$ and $y^2 = xf_c(x)$.
\begin{Theorem}\label{strongselm}
Let $\Sigma \subset \Inv(\Z)$ be any large family with the property that $\Sigma_2$ is contained in the subset of $\Inv(\Z_2)$ consisting of all $c = (a_1, \hdots a_{n-1},e)$ such that $2^{4i}|a_i$ and $2^{2n}|e$. Then the average size over $\Sigma$ of $\Sel_2(J_{1,c})$ is bounded above by 6, and the average size of $\Sel_{(1,2)}(c)$ equals 2. 
\end{Theorem}

Theorems \ref{2selm} and \ref{12selm} follow from by applying Theorem \ref{strongselm} to the large family $\Sigma^0$ chosen as follows: 

For $p \neq 2$, let
$\Sigma_p \subset \Inv(\Z_p)$ consist of all $c = (a_1, \hdots
a_{n-1},e)$ such that either $p^2i \nmid a_i$ for some $i$, or $p^n
\nmid e$. Let $ \Sigma_2$ consist of all $c = (a_1, \hdots a_{n-1},e)$
such that $2^{4i} | a_i$, $2^{2n}| e$, and either $2^{6i} \nmid a_i$
for some $i$, or $2^{3n} \nmid e$. Let $\Sigma^0 \subset \Inv(\Z)$ be
the subset defined by the local condition $\Sigma_p$, i.e. $\Sigma^0 =
\{c \in \Inv(\Z)|\ \forall p, \ c\in \Sigma_p \}$. 

We spend the rest of the section proving Theorem \ref{strongselm}. Henceforth, by soluble we will mean 1-soluble. 
Let $\phi: \Vij(\Z) \rightarrow [0,1]$ be the function
\begin{equation}\label{sol}
\phi(T)=
\begin{cases}
\Big(\displaystyle{\sum_{T'}\frac{\#\Stab_{T'}(\Q)}{\#\Stab_{T'}(\Z)}\Big)^{-1}} & \text{if T is locally soluble,  and T has invariants in $\Sigma$},\\
0, & \text{otherwise}
\end{cases}
\end{equation}
where the sum is over a complete set of representatives for the action of $\Gij(\Z)$ on the $\Gij(\Q)$-equivalence class of $T$ in $\Vij(\Z)$. 
Similarly, let $\wij: \Vij(\Z) \rightarrow [0,1]$ be defined as follows: 
\begin{equation}\label{sol12}
\wij(T)=
\begin{cases}
\Big(\displaystyle{\sum_{T'}\frac{\#\Stab_{T'}(\Q)}{\#\Stab_{T'}(\Z)}\Big)^{-1}} & \text{if T is locally (1,2)-soluble, and T has invariants in $\Sigma$},\\
0, & \text{otherwise}
\end{cases}
\end{equation}
The sum is again over a complete set of representatives for the action of $\Gij(\Z)$ on the $\Gij(\Q)$-equivalence class of $Y$ in $\Vij(\Z)$. 

\begin{Proposition}
Let $\Sigma^{(a,b)} = \Sigma \cap \Inv(\R)^{(a,b)}$. Then,
$$\sum_{\substack{c \in \Sigma^{(a,b)} \\\h(c)\le X}}(\# \Sel_2(J_{1,c}) -2) = N_{\phi}(\Vij(\Z)^{(a,b)},X) + o(X^{n^2}),$$
$$\sum_{\substack{c \in \Sigma^{(a,b)} \\\h(c)\le X}}(\# (\Sel_2(J_{1,c})\cap \Sel_2(J_{2,c})) -2) = N_{\wij}(\Vij(\Z)^{(a,b)},X) + o(X^{n^2})$$
\end{Proposition}
\begin{proof}
The proof in the (1,2)-soluble case is identical to the proof of the soluble case, hence we will demonstrate only the first proof. By Proposition \ref{difdist}, the set of $c\in \Inv(\Z)^{(a,b)}$ with $\h(c) \le X$ and which satisfy the condition that the two distinguished elements lie in the same $\Gij(\Q)$ orbit, is $o(X^{n^2})$. Further, by Theorem \ref{intorb} every element in $\Sel_2(J_{1,c}) \subset H^1(J_{1,c}[2])$ gives rise to integral orbits, which by definition are locally soluble. Therefore,
$$\sum_{\substack{c \in \Sigma^{(a,b)} \\\h(c)\le X}}(\# \Sel_2(J_{1,c}) -2) = \#(\Gij(\Q) \backslash \Vij(\Z)^{\irr, \phi \neq 0}_{\h\le X}) + o(X^{n^2}).$$
Recall that in the definition of $N(S,X)$ for sets $S$, the $\Gij(\Z)$-orbit of $T \in S$ was weighted by $\frac{1}{\#\Stab_{T}(\Z)}$. Let $c\in \Inv$ be the invariants of $T \in \Vij(\Z)^{(a,b), \phi \neq 0}$. Suppose that $T=T_1 \hdots T_k$ are a set of representatives for the action of $\Gij(\Z)$ on the $\Gij(\Q)$-equivalence class of $T$ in $\Vij(\Z)$. Each $T_i$ would be counted on the right hand side with a weight of $\frac{\phi(T_i)}{\#\Stab_{T_i}(\Z)}$. The term $\phi(T_i)$ is independent of $i$, and equals $\Big(\sum_{i=0}^k\frac{\#\Stab_{T_i}(\Q)}{\#\Stab_{T_i}(\Z)}\Big)^{-1}$. We now sum over $i$:
$$\sum_{i=0}^k\frac{\phi(T_i)}{\#\Stab_{T_i}(\Z)} = \phi(T)\sum_{i=0}^k\frac{1}{\#\Stab_{T_i}(\Z)} = \frac{1}{\#\Stab_{T}(\Q)}.$$

Therefore, $N_{\phi}(\Vij(\Z)^{(a,b)},X)$ counts each locally soluble $\Gij(\Q)$-orbit having invariants in $\Sigma$, with a representative $T$ in $\Vij(\Z)$ with a weight of $\frac{1}{\#\Stab_{T}(\Q)}$. By Proposition \ref{nostab}, the number of orbits having non-trivial stabilizer over $\Q$ is $o(X^{n^2})$. The proposition follows. 
\end{proof}

We define the local analogues $\phi_p$ and $\wijp$, functions from $\Vij(\Z_p) \rightarrow [0,1]$. 
 \begin{equation}\label{psol}
\phi_p(T)=
\begin{cases}
\Big(\displaystyle{\sum_{T'}\frac{\#\Stab_{T'}(\Q_p)}{\#\Stab_{T'}(\Z_p}\Big)^{-1}} & \text{if T is soluble, and T has invariants in $\Sigma$},\\
0, & \text{otherwise}
\end{cases}
\end{equation}
\begin{equation}\label{psol12}
\wijp(T)=
\begin{cases}
\Big(\displaystyle{\sum_{T'}\frac{\#\Stab_{T'}(\Q_p)}{\#\Stab_{T'}(\Z_p)}\Big)^{-1}} & \text{if T is (1,2)-soluble, and T has invariants in $\Sigma$},\\
0, & \text{otherwise}
\end{cases}
\end{equation}
where in both cases, the sum is over a complete set of representatives of $\Gij(\Z_p)$ on the $\Gij(\Q_p)$-equivalence class of $Y$ in $\Vij(\Z_p)$. The local weight functions are related to the global ones in the following way. 
\begin{Proposition}
Let $w$ denote either $\phi$ or $\wij$. Then $\displaystyle{w(T) = \prod_pw_p(T)}$.
\end{Proposition}
\begin{proof}
The class numbers of $\SO(V_1)$ and $\SO(V_2)$ are $1$, and therefore the class number of $\Gij = \SO(V_1) \times \SO(V_2)$ is also 1. This being the case, the same proof as in \cite{BS2selm} applies. 
\end{proof}

Recall that for $c\in \Inv(\Q_p)$, $\Vij_c$ is the fiber in $\Vij$ over $c$. In order to compute $N_{\phi}(\Vij(\Z),X)$ and $N_{\wij}(\Vij(\Z),X)$, we will need to compute the $p$-adic integrals listed in Theorem \ref{infcong}. To that end, let $dT$ and $dc$ denote Euclidean mesaures on $\Vij$ and $\Inv$, so that $\Vij(\Z)$ and $\Inv(\Z)$ have covolume 1. Pick $\omega$, an algebraic differential form that generates the rank 1 module of top-degree left invariant differential forms on $\Gij = \SO(V_1) \otimes \SO(V_2)$. We cite the following result from \cite[Proposition 3.11]{BS2selm}. 

\begin{Proposition}\label{pvol}
Let $|.|$ denote the $p$-adic valuation on $\Z_p$. There then exists a rational nonzero constant $\mathcal{J}$, independent of $p$, such that for any $\Gij(\Z_p)$ invariant function $w_p$ on $\Vij(\Z_p)$, we have 
$$\int_{\Vij(\Z_p)}w(T)dT = \Vol(\Gij(\Z_p))|\mathcal{J}|\int_{c\in \Inv(\Z_p)}\Big(\sum_{T \in \frac{\Vij_c(\Z_p)}{\Gij(\Z_p)}}\frac{w_p(T)}{\#\Stab_{T}(\Z_p)}\Big)dc$$
where $\frac{\Vij_c(\Z_p)}{\Gij(\Z_p)}$ denotes a set of representatives for the action of $\Gij(\Z_p)$ on $\Vij_c(\Z_p)$.
\end{Proposition}

We will also want to express the volume of $\Fu\Dab(X)$ in terms of the volume on $\Gij(\R)$ and $\Inv(\R)$. The proof of the following proposition is the same as in \cite[Proposition 3.12]{BS2selm}. 
\begin{Proposition}\label{rvol}
The volume of the multiset $\Fu\Dabs(X)$ is given by
$$\Vol(\Fu\Dabs(X)) = \tau^{(a,b)}_{\sol}|\mathcal{J}|\Vol(\Fu)\Vol(\Inv(\R)^{(a,b)}_{\h< X}).$$ 
Here, $\mathcal{J}$ is the same constant that appears in Proposition \ref{pvol}, and the numbers $\tau^{(a,b)}_{\sol}$ are as in \S 7.1.
\end{Proposition}
The proofs of \cite[Propositions 3.11, 3.12]{BS2selm} apply, because the action of $\Gij$ on $\Vij$ satisfy the conditions in \cite[Remark 3.14]{BS2selm}. Indeed, the ring of invariants is freely generated; the stabilizer of a regular semisimple element is a finite group scheme of order $2^{n-1}$ and is therefore uniformly bounded (outside the discriminant-zero locus); the sum of the degrees of the invariants equals $n^2$, the dimension of $\Vij$; and there exist Kostant sections $\kappa: \Inv \rightarrow \Vij$. 

We need to simplify the expression
$$\int_{c \in \Sigma_p}\Big(\sum_{T \in \frac{\Vij_c(\Z_p)}{\Gij(\Z_p)}}\frac{w_p(T)}{\#\Stab_{T}(\Z_p)}\Big)dc$$
where $w$ stands for either $\phi_p$, or $\wijp$. We have
$$\sum_{T \in \frac{\Vij_c(\Z_p)}{\Gij(\Z_p)}}\frac{w_p(T)}{\#\Stab_{T}(\Z_p)} = \frac{\#(\Vij_c(\Q_p)_{\sol}/\Gij(\Q_p))}{\#\Stab_{T}(\Q_p)},$$
where the subscript $_{\sol}$ is as in \S 7.1. Depending on whether $\sol$ stands for 1-soluble, or (1,2)-soluble, we have 
$$\displaystyle{\#(\Vij_c(\Q_p)_1/\Gij(\Q_p)) = \#\frac{J_{1,c}(\Q_p)}{2J_{1,c}(\Q_p)}}$$
and
$$\displaystyle{\#(\Vij_c(\Q_p)_{(1,2)}/\Gij(\Q_p))= \# \bigg(\frac{J_{1,c}(\Q_p)}{2J_{1,c}(\Q_p)} \displaystyle{\cap} \frac{J_{2,c}(\Q_p)}{2J_{2,c}(\Q_p)}\bigg)}$$
where the intersection happens in $H^1(\Q_p,\Stab_c)$.

\subsection{Average size of the 2-Selmer group}
We now bound the average size of the 2-Selmer group of our family of hyperellictic curves. Recall that the set of hyperelliptic curve with the extra marked points is in bijection with $\Sigma \subset \Inv(\Z)$. Without any loss of generality, we restrict ourselves to the family over $\Sigma^{(a,b)}$ for a fixed pair $(a,b)$. Using Propositions \ref{pstoll} and \ref{2stoll}, if $c \in \Sigma_p$, the integrand in Proposition \ref{pvol} is constant, and equals $b_p$. The integrand is zero if $c \notin \Sigma_p$. Substituting this in Proposition \ref{pvol}, we obtain
$$\sum_{\substack{c \in \Sigma^{(a,b)} \\\h(c)\le X}}(\#\Sel_2(J_1(C))-2) \leq N(\Vij(\Z)^{(a,b)},X)\prod_{p}\big(b_p|\mathcal{J}|_p\Vol(\Gij(\Z_p))\Vol(\Sigma_p)\big) + o(X^{n^2}).$$
Further, by Proposition \ref{rvol} and Theorem \ref{countint}, 
$$N(\Vij(\Z),X) = |\mathcal{J}|\frac{\tau^{(a,b)}_1}{\tauab}\Vol(\Fu)\Vol(\Inv(\R)^{(a,b)}_{\h<X}) + o(X^{n^2}).$$  Clearly, $b_{\infty} = \frac{\tau^{(a,b)}_1}{\tauab}$, because the numerator equals the number of real soluble orbits. Therefore, by using the fact that the $b_{\nu}$ all multiply to 1, we see that the main term simplifies to 
$$\displaystyle{\Vol(\Inv(\R)^{(a,b)}_{\h<X}) \Vol(\Fu)\prod_p\Vol(H(\Z_p))}.$$
The product of the local volumes along with $\Vol(\Fu)$ is the Tamagawa number of $H$, which equals $4$ (\cite{Langlands}). Therefore, we have 
$$\displaystyle{\sum_{\substack{c \in \Sigma^{(a,b)} \\\h(c)\le X}}(\#\Sel_2(J_{1,c})-2) \leq 4\Vol(\Inv(\R)^{(a,b)}_{\h<X})\prod_p\Vol(\Sigma_p)}.$$
The number of hyperelliptic curves in our family with height less than $X$ is $\displaystyle{\sum_{\substack{c \in \Sigma^{(a,b)} \\\h(c)\le X}}} 1$, which by Theorem \ref{last} is  $\displaystyle{\Vol(\Inv(\R)^{(a,b)}_{\h<X}) \prod_p\Vol(\Sigma_p)}$ up to an error of $o(X^{n^2})$. Putting all this together, we have 
$$\displaystyle{ \lim_{X \to \infty}}\frac{\displaystyle{\sum_{\substack{c \in \Sigma^{(a,b)} \\\h(c)\le X}}(\#\Sel_2(J_{1,c})-2)}}{\displaystyle{\sum_{\substack{c \in \Sigma^{(a,b)} \\\h(c)\le X}} 1}}\le 4.$$
Therefore, we have the average size of the 2-Selmer group is bounded above by 6. 

\subsection{The (1,2)-Selmer}

\begin{Proposition}\label{diverges}
Let $p>2$ be a large enough prime. We have 
$$\int_{\Inv(\Z_p)}\displaystyle{\frac{\#(J_{1,c}(\Q_p)/2J_{1,c}(\Q_p) \cap J_{2,c}(\Q_p)/2J_{2,c}(\Q_p))dh}{\# \Stab_c(\Q_p)}}dc \le (1-a/p)\Vol(\Inv(\Z_p))$$
where the intersection happens inside $H^1(\Q_p,\Stab_c)$, and $a$ is some positive constant independent of $p$.
\end{Proposition}
\begin{proof}
For ease of notation, we will drop the subscript $c$. We had remarked earlier (Proposition \ref{pstoll}) that $\displaystyle{\frac{\#J_1(\Q_p)/2J_1(\Q_p)}{\#\Delta(\Q_p)}} = 1$. The same holds true with $J_2$ in place of $J_1$. Therefore, $1$ is a trivial upper-bound for the integral. We note that for some $c$, if the images of $J_1(\Q_p)/2J_1(\Q_p)$ and $J_2(\Q_p) /2J_2(\Q_p)$ don't coincide, then the integrand will be at most $1/2$. We will show that there exists $S \subset \Inv(\Z_p)$ of volume greater than $r/p$, such that for $c \in S$, the images of $J_1$ and $J_2$ don't coincide. Here, $r$ is the constant alluded to in Lemma \ref{smallonetwo}. The proposition follows from the existence of $S$. Indeed, 
$$\int_{\Inv(\Z_p)}\displaystyle{\frac{\#(J_{1,c}(\Q_p)/2J_{1,c}(\Q_p) \cap J_{2,c}(\Q_p)/2J_{2,c}(\Q_p))dh}{\# \Stab_c(\Q_p)}}dc \le$$
$$\int_{\Inv(\Z_p) \setminus S}1dc + \int_{c\in S}1/2dc \le (1- r/p)\Vol(\Inv_p(F)) + (r/2p)\Vol(\Inv_p(F))$$
Setting $a = r/2$ gives the proposition. 

Let $S_p\subset \Inv(\F_p)$ be the subset defined by the conditions in Lemma \ref{smallonetwo}. Let $S' \subset \Inv(\Z_p)$ be the set of all points reducing to $S_p$. For $c \in S'$, the polynomial $f_c$ factors into distinct linear factors over $\Z_p$ (Hensel's lemma), and so the discriminant of $f_c$ is not zero. In fact, it is a unit. For such $c$, exactly one of the roots of $f_c$ is a multiple of $p$. Indeed, its $p$-adic valuation has to equal $2b$ for some positive integer $b$ (this is because $f_c(0) = e_c^2$). Let $S''$ denote the set of all $c \in S'$ such that $e_c = 0$. Clearly, $S''$ a measure-zero set. Let $S = S' \setminus S''$. The set $S$ has volume at least $r/p$. 

We therefore are left with showing that for $c \in S$, the images of $J_1(\Q_p)$ and $J_2(\Q_p)$ in $H^1(\Q_p, \Stab_c)$ do not coincide. By construction $f_c$ splits into linear factors which are pairwise unequal modulo $p$ over $\Z_p$. We therefore have that $\Stab_c = \Res_{\Q_p^n / \Q_p} (\mu_2) _{N = 1}$, and $H^1(\Q_p,\Stab_c) = ((\Q_p^{\times})^{n}/(\Q_p^{\times 2})^{n})_{N=1}$. Because the discriminant of $f_c$ is a $p$-adic unit, by Proposition \ref{vtemp} the image of $J_1(\Q_p)$ equals $((\Z_p^{\times})^{n}/(\Z_p^{\times 2})^{n})_{N=1}$. In order to show that the image of $J_2(\Q_p)$ isn't the same, it suffices to show the existence of an element in $H^1(\Q_p,\Stab_c)$ with odd p-adic valuation in at least one of its components. We will show that $pf_c(p)$ is a perfect square in $\Q_p$, thereby demonstrating the existence of a $\Q_p$-rational point $Q$ of $y^2 = xf_c(x)$ with $x$-coordinate having p-adic valuation $1$ (indeed, the $x$ coordinate by construction would equal $p$). It is then easy to see using the explicit descent map described in \cite{Stoll}, that the image of $Q - \infty_1$ in $H_1(\Q_p,\Stab_c)$ has the property that at least one of its components has odd p-adic valuation, thereby completing the proof. 

Thus it suffices to show that $pf_c(p)$ is a perfect square in $\Q_p$. Suppose that $u'_i \in \Z_p^{\times}$, which lift $u_i \in \F_p^{\times}$, are the roots of $f_c$. Suppose that $p^{2b}u'_e \neq 0$ is the final root (which as mentioned above has $p$-adic valuation equaling $2b$). Then, $f_c(p) = (p-p^{2b}u'_e)\prod_{i=1}^{n-1}(p-u'_i) =  p(1-p^{2b-1}u'_e)\prod_{i=1}^{n-1}(p-u'_i)$. Since $n$ is odd, 
$$\prod_{i=1}^{n-1}(p-u'_i) = \prod_{i=1}^{n-1}(u'_i-p) \equiv \prod_{i=1}^{n-1}u'_i \mod{p}.$$
By construction, $\displaystyle{\prod_{i=1}^{n-1}(p-u'_i)}$ is a non-zero square modulo $p$, and is therefore a square in $\Z_p ^{\times}$. Similarly, $1 - p^{2b-1}u'_e$ is also a square in $\Z_p^{\times}$. Therefore, $f_c(p)$ is $p$ multiplied by a square, and so $pf_c(p)$ must be a square in $\Q_p^{\times}$. We have proved our result. 
\end{proof}

We now prove that the average size of $\Sel_2(J_1)\cap \Sel_2(J_2)$ is 2. Again, it suffices to restrict ourselves to the family over $\Sigma^{(a,b)} = \Sigma \cap \Inv(\R)^{(a,b)}$ for a fixed pair $(a,b)$. It suffices to prove that 
$$\lim_{X \to \infty}\frac{\displaystyle{\sum_{\substack{c \in  \Sigma^{(a,b)} \\\h(c)\le X}}(\#(\Sel_2(J_{1,c})\cap \Sel_2(J_{2,c}))-2)}}{\displaystyle{\sum_{\substack{c \in  \Sigma^{(a,b)} \\\h(c)\le X}}} 1} =0.$$
The denominator (up to an error of $o(X^{n^2})$) is a fixed constant multiple of $X^{n^2}$. We will prove that the numerator is $o(X^2)$. 
Indeed, we have that 
$$\sum_{\substack{c \in  \Sigma^{(a,b)} \\\h(c)\le X}}(\#(\Sel_2(J_1(C))\cap \Sel_2(J_2(C)))-2) = N_{\wij}(\Vij(\Z)_{(1,2)}^{(a,b)},X) + o(X^{n^2}).$$
By Theorem \ref{infcong}, we have 
$$N_{\wij}(\Vij(\Z)_{(1,2)}^{(a,b)},X) \le  N(\Vij(\Z)^{(a,b)}_{(1,2)},X)\prod_p\int_{\Vij(\Z_p)}\wijp(T)dT + o(X^{n^2}).$$
The above proposition shows that the product of the local weights converges to $0$. We have thus shown that the expression is dominated by the error term $o(X^{n^2})$. Therefore, the above limit equals zero and we have proved Theorem \ref{strongselm}.

\AtEndDocument{{\footnotesize%
  \textit{E-mail address}, \texttt{ashankar@math.harvard.edu} \par
  \textsc{Department of Mathematics, Harvard University, Cambridge, MA 02138,} \par \textsc{USA}
}}

\end{document}